\documentclass[11pt]{article}
\usepackage{geometry,graphicx,latexsym,amssymb,verbatim, multicol}
\usepackage{amsmath, amssymb, amsthm}
\usepackage[usenames,dvipsnames]{color}
\usepackage{tikz}
\usetikzlibrary{calc}
\usepackage{hyperref}

\newtheorem{thm}{Theorem}[section]
\newtheorem{lem}[thm]{Lemma}

\newtheorem{prob}{Problem}
\newtheorem{cor}[thm]{Corollary}

\newtheorem{ex}[thm]{Example}
\newtheorem{rem}[thm]{Remark}

\renewcommand{\comment}[1]{}

\renewenvironment{proof}{\noindent {\it Proof.}}{$\Box$\\}
\newcommand{\dontshow}[1]{}

\newcommand{\ov}{\overline}

\begin{document}

\begin{center}
{\LARGE Partial domination - the isolation number of a graph}
\mbox{}\\[8ex]

\begin{multicols}{2}

Yair Caro\\[1ex]
{\small Dept. of Mathematics and Physics\\
University of Haifa-Oranim\\
Tivon 36006, Israel\\
yacaro@kvgeva.org.il}

\columnbreak

Adriana Hansberg\\[1ex]
{\small Instituto de Matem\'aticas\\
UNAM Juriquilla\\
Quer\'etaro, Mexico\\
ahansberg@im.unam.mx}\\[2ex]

\end{multicols}

\end{center}

\begin{abstract}

We prove the following result: If $G$ be a connected graph on $n \ge 6$ vertices, then there exists a set of vertices $D$ with $|D| \le \frac{n}{3}$ and such that $V(G) \setminus N[D]$ is an independent set, where $N[D]$ is the closed neighborhood of $D$. Furthermore, the bound is sharp. This seems to be the first result in the direction of partial domination with constrained structure on the graph induced by the non-dominated vertices, which we further elaborate in this paper.
 \\

\end{abstract}

\section{Introduction}

Many variants of the basic topic of domination can be formulated as follows. Let $G$ be a graph and $D$ a set of vertices of $G$ such that  $D \cup  N(D) = V(G)$,  where $N(D)$ is the set of neighbors of vertices in $D$ that do not belong to $D$. Given two graph properties $\mathcal{P}$ and $\mathcal{Q}$, we say that $D$ is a $(\mathcal{P,Q})$-dominating set if the graph induced by $D$ has property $\mathcal{P}$ and the graph induced by $N(D)$ has property $\mathcal{Q}$.
 
Among the best known properties $\mathcal{P}$ are: \emph{connected domination}, where the graph induced by $D$ is connected (see \cite{DHH, KSKW, CWY}), \emph{$k$-connected domination}, where the graph induced by $D$ is $k$-connected (see \cite{LWAB, TZTX}), \emph{total domination} where the graph induced by $D$ has no isolates (see \cite{He, HeYe}), \emph{paired domination}, where the graph induced by $D$ has a perfect matching (see \cite{DeHe}), and many more.
 
Among the best known properties $\mathcal{Q}$ are: \emph{$k$-domination}, where $|N(u)\cap D| \ge k$ for every $u \in N(D)$ (see \cite{Ha, FaHaVo}), \emph{locating domination}, where  $N(u) \cap D \neq N(v) \cap D$ for any $u ,v \in N(D)$ (see \cite{Sla}), \emph{fair domination}, introduced in \cite{CaHaHe}, where $|N(u) \cap D| =| N(v) \cap D|$ for any $u ,v \in N(D)$, and many more.

There are also variants of domination which set a condition on both $D$ and $N(D)$, like the \emph{$k$-tuple domination}, where $|N[u]\cap D| \ge k$ for every $u \in V(G)$ (see \cite{CoTh, GaZv, RaVo}), or the \emph{identifying codes}, where $\emptyset \neq N(u) \cap D \neq N(v) \cap D \neq \emptyset$ for any $u ,v \in V(G)$ (see \cite{Fou}). See also the survey \cite{ChFaHaVo} for more information on multiple domination parameters.

Here we are going to present, for the first time to our knowledge, the following extension of this model to the case that $D \cup N(D) \neq V(G)$, that is, we do not necessarily assume that $D$ is a dominating set. 

The reason for doing so is because there may be situations in which we do not necessarily need to dominate all the vertices in a graph. Les us illustrate this by the following example. Suppose $G$ represents a communication network, and a security agency wants to detect every conversation between two members, i.e. two adjacent vertices, in the network. We can imagine some centers forming the vertices of $D$ which can listen to all nodes in $N[D]$ and, as long as $V \setminus N[D]$ is an independent set, they can still listen to all conversations having located microphones in all vertices of $D$.  
 
 
Let $G$ be a graph, $D \subseteq V(G)$ a set, $N(D)$ the neighborhood of $D$ and $R(D) = V(G) \setminus N[D]$  the remainder with respect to $D$.  When $R(D) = \emptyset$, $D$ is a dominating set. 
 
So, in addition to the properties $\mathcal{P}$ and $\mathcal{Q}$ imposed on $D$ and $N(D)$, we would like to impose a property $\mathcal{R}$ on $R(D)$, where $\mathcal{R}$ can be the property of $R(D)$ being a $k$-independent set or inducing a forest, a planar graph, a $K_k$-free graph, or a $k$-colorable graph, to mention some examples.
 
Problems involving partitions of this type $D \cup N(D) \cup  R(D) = V(G) $ with D having property $\mathcal{P}$, $N(D)$ having property $\mathcal{Q}$ and $R(D)$ having property $\mathcal{R}$ are called constrained partial domination problems.  
 
In this paper, we shall consider simple constrained partial domination problems with no restriction on properties $\mathcal{P}$ or $\mathcal{Q}$ but with $R(D)$ having certain property $\mathcal{R}$ as mentioned above: being a $k$-independent set, or inducing a forest, a planar graph, a $K_k$-free graph or a $k$-colorable graph, etc.

\subsection{Notation}

For notation and Graph Theory terminology we in general
follow~\cite{West}. Specifically, let $G = (V, E)$ be a simple graph with
vertex set $V = V(G)$ of order~$n(G) = |V|$ and edge set $E = E(G)$ of size~$m(G) =|E|$. For a subset $S \subseteq V$, the \emph{open neighborhood} of $S$ is the set $N_G(S) = \{u \in V \setminus S \, | \, uv \in E, v \in S\}$, while the \emph{closed neighborhood} of $S$ is the set $N_G[S] = N_G(S) \cup S$. When $S = \{v\}$, we write $N_G(v)$ and $N_G[v]$ for the open and closed neighborhoods of $v$, respectively. The \emph{degree} of a vertex $v$ is $\deg_G(v) = |N_G(v)|$. When the graph is clear from the context, we may write $N(S), N[S], N(v), N[v], \deg(v)$. The \emph{minimum degree} $\delta(G)$ of a graph $G$ is the minimum among all vertex degrees of $G$. Likewise, the \emph{maximum degree} $\Delta(G)$ of $G$ is the maximum among all vertex degrees of $G$.  We call a vertex $v$ \emph{isolate} or say it is \emph{isolated} in $G$ if $\deg_G(v) = 0$. Moreover, a vertex $v$ is called \emph{leaf} if $\deg_G(v) = 1$.

The \emph{complement} $\overline{G}$ of a graph $G = (V,E)$ is the graph consisting of the vertex set $V$ and all edges between vertices of $V$ that do not belong to $E$. Given a subset $S \subseteq V$, the graph $G[S]$ denotes the subgraph of $G$ induced by $S$ and, for an integer $t \ge 1$, $tG$ denotes the graph consisting of $t$ vertex disjoint copies of $G$. Given a graph $H$, we say that  $G$ is \emph{$H$-free} if $G$ does not contain $H$ as a subgraph (observe that $H$ has not to be necessarily induced). 

The \emph{complete graph} on $n$ vertices is denoted by $K_n$ and we write $K_{p,q}$ for the \emph{complete bipartite graph} with partition sets of $p$ and $q$ vertices. Further, the \emph{cycle} and the \emph{path} on $n$ vertices are written as $C_n$ and $P_n$, respectively. We define the {\it cartesian product} of two graphs $G_1$ and $G_2$ as the graph $G_1 \times G_2$ with vertex set $V(G_1)  \times V(G_2)$ and such that two vertices $(u_1,u_2)$ and $(v_1,v_2)$ are adjacent if and only if either $u_1 = v_1$ and $u_2v_2 \in E(G_2)$ or $u_2 = v_2$ and $u_1v_1 \in E(G_1)$. For $t, s \ge 1$, the graph $G = P_s \times P_t$ is called \emph{grid graph}.

We call a subset $S \subseteq V$ \emph{dominating} in $G = (V,E)$ if $|N(v) \cap S| \ge 1$ for all $v \in V \setminus S$. The minimum cardinality of a dominating set of $G$ is denoted by $\gamma(G)$. On the other hand, $S$ is called \emph{independent} if $G[S]$ is the empty graph and with $\alpha(G)$ we denote the maximum cardinality of an independent set of $G$. For $k \ge 0$, $S$ is a $k$-independent set if $\Delta(G[S]) \le k$ and $\alpha_k(G)$ denotes the maximum cardinality of a $k$-independent set of $G$. Observe that $0$-independent is the same as independent.

\subsection{The isolation number of a graph} 

We will now introduce a more formal definition and language for the main protagonist of this paper: the isolation number of a graph.

Let $G = (V,E)$ be a graph and $\mathcal{F}$ a family of graphs. We call a set of vertices $S \subseteq V$ \emph{$\mathcal{F}$-isolating} if the graph induced by the set $R(S) = V \setminus N[S]$ contains no member of $\mathcal{F}$ as a subgraph.

\begin{ex} Consider the following families $\mathcal{F}$ of graphs.
\begin{enumerate}
\item[(1)] If $\mathcal{F} = \{ K_1 \}$, then an $\mathcal{F}$-isolating set coincides with the usual definition of a dominating set.
\item[(2)] If $\mathcal{F} = \{K_2\}$, then the vertices not dominated by the $\mathcal{F}$-isolating set form an independent set.
\item[(3)] If $\mathcal{F} = \{K_{1,k+1}\}$ for an integer $k \ge 0$, then the set of vertices not dominated by the $\mathcal{F}$-isolating set induces a graph of maximum degree at most $k$ or, in other words, these vertices form a $k$-independent set (see \cite{CaHa, ChFaHaVo} for more information and recent results).
\item[(4)] If $\mathcal{F} = \{C_k \; | \; k \ge 3\}$, then the vertices not dominated by the $\mathcal{F}$-isolating set induce a forest.
\item[(5)] If $\mathcal{F}$ is the family of all trees on $k$ vertices, then the non-dominated vertices form a subgraph whose components have all order at most $k-1$.
\end{enumerate}
\end{ex}

The minimum cardinality of an $\mathcal{F}$-isolating set of a graph $G$ will be denoted $\iota(G, \mathcal{F})$ and called the \emph{$\mathcal{F}$-isolation number} of $G$. When $\mathcal{F} = \{H\}$, we will set $\iota(G,\mathcal{F}) = \iota_H(G)$. In case $\mathcal{F} = \{K_{1,k+1}\}$, we shall use the notation $\iota_k(G)$, and when $k = 0$ we will write for short $\iota(G)$ instead of $\iota_0(G)$. Finally, $\iota_k(G)$ will be called the \emph{$k$-isolation number} and $\iota(G)$ just \emph{isolation number}.

\subsection{Structure of the paper}

The paper is organized as follows:

In Section 2 we give some basic properties and examples concerning  $\iota(G,\mathcal{F})$, relating $\mathcal{F}$-isolating sets to dominating sets, as well as some concrete constructions that will be useful in later sections. Later on we will mainly deal with $\iota(G)$ and sometimes $\iota_k(G)$.
 
In Section 3, we consider upper bounds on $\iota(G)$ and $\iota_k(G)$  in terms of order, maximum degree and minimum degree with special emphasis on $G$ being a connected graph, and prove some sharpness results as well.

In Section 4, we consider lower bounds on $\iota(G)$ and $\iota_k(G)$ in terms of average degree, maximum degree and minimum degree and prove some sharpness results.

In Section 5, we consider some classes of graphs such as trees, maximal outerplanar graphs, claw-free graphs and grid graphs, and compute some sharp results for $i(G)$ for graphs in these families.

In Section 6, we deal with Nordhaus-Gaddum type results for $\iota(G)+\iota(\overline{G})$  and prove sharp upper bounds.

In the closing section 7, we introduce some open problems for further research.\\

\section{Basic examples and facts}

\begin{ex}
Consider the following examples.
\begin{enumerate}
\item[(1)] For $k \le n-1$, $\iota_k(\overline{K_n}) = 0$ and $\iota_k(K_n) = 1$.
\item[(2)] $\iota(C_n) = \lceil n/4\rceil$, $\iota_1(C_n) = \lceil \frac{n}{5} \rceil $, $\iota(P_n) = \lceil (n-1)/4 \rceil$ and $\iota_1(P_n) = \lceil \frac{n-2}{5} \rceil$.
\item[(3)] If $P$ is the Petersen-graph, $\iota(P) =  3$, $\iota_1(P) = 2$ and $\iota_2(P) = 1$.
\item[(4)] $\iota$ is a non-monotone parameter with respect to edge-deletion: \[\iota(K_5) = 1  < \iota(C_5) = 2 > \iota(P_5) = 1.\]
\end{enumerate}
\end{ex}
%
%

\begin{lem}\label{la-family}
Let $G$ be a graph on $n$ vertices and $\mathcal{F}$ and $\mathcal{F}'$ be two families of graphs. The following assertions hold.
\begin{enumerate}
\item[(i)] If $\mathcal{F}' \subseteq \mathcal{F}$, then $\iota(G,\mathcal{F}') \le \iota(G, \mathcal{F})$.
\item[(ii)] If $F_1, F_2 \in \mathcal{F}$ such that $F_1 \subseteq F_2$ and $\mathcal{F}' = \mathcal{F} \setminus \{F_2\}$, then $\iota(G, \mathcal{F}) = \iota(G,\mathcal{F}')$.
\item[(iii)] If for all $F' \in \mathcal{F}'$ there is an $F \in \mathcal{F}$ such that $F \subseteq F'$, then $\iota(G, \mathcal{F}') \le \iota(G,\mathcal{F})$.
\end{enumerate}
\end{lem}

\begin{proof} 
(i) Let $\mathcal{F}' \subseteq \mathcal{F}$ and let $S$ be a minimum $\mathcal{F}$-isolating set of $G$. Then, clearly, $S$ is an $\mathcal{F}'$-isolating set of $G$ and thus $\iota(G,\mathcal{F}') \le \iota(G, \mathcal{F})$.\\
(ii) By item~(i), the inequality $\iota(G, \mathcal{F}') \le \iota(G,\mathcal{F})$ is clear. So let $S$ be a minimum $\mathcal{F}'$-isolating set of $G$. Then $G - N[S]$ is $\mathcal{F}'$-free and thus, in particular, $F_1$-free. Since $F_1$ is a subgraph of $F_2$, $G - N[S]$ has to be $F_2$-free, too. Hence, $G - N[S]$ is $\mathcal{F}$-free and we obtain $\iota(G, \mathcal{F}) \le |S| = \iota(G,\mathcal{F}')$. Now the both inequalities imply the result.\\
(iii) Let $S$ be a minimum $\mathcal{F}$-isolating set of $G$. Then $G - N[S]$ is $\mathcal{F}$-free. Since for all $F' \in \mathcal{F}'$ there is an $F \in \mathcal{F}$ such that $F \subseteq F'$, $G - N[S]$ is also $\mathcal{F}'$-free. Hence, $S$ is an $\mathcal{F}'$-isolating set of $G$ and thus $\iota(G, \mathcal{F}') \le \iota(G,\mathcal{F})$.
\end{proof}

In the sequel of this paper and in view of Lemma~\ref{la-family}~(ii), we will consider only families of graphs without inclusions among its members.

\begin{lem}\label{la-dom-sum}
Let $G$ be a graph on $n$ vertices and $\mathcal{F}$ a family of graphs. The following assertions hold.
\begin{enumerate}
\item[(i)] $\iota(G,\mathcal{F}) \le \gamma(G)$.
\item[(ii)]$\iota(G, \mathcal{F}) = \min \{ \iota(G[A], \mathcal{F}) + \gamma(G[B]) \;|\; A \cup B \mbox{ is a partition of } V \}$.
\item[(iii)] $\iota(G, \mathcal{F}) = \min \{ \gamma(G \setminus H) \;|\; H \mbox{ is an induced $\mathcal{F}$-free subgraph of } $G$ \}$.
\end{enumerate}
\end{lem}

\begin{proof}
(i) Since every dominating set is also an $\mathcal{F}$-isolating set, clearly $\iota(G, \mathcal{F}) \le \gamma(G)$.\\
(ii) Let $V = A \cup B$ be a partition of $V$. Let $I$ be a minimum $\mathcal{F}$-isolating set of $G[A]$ and $D$ a minimum dominating set of $G[B]$. Then $I \cup D$ is an $\mathcal{F}$-isolating set of $G$ and hence $\iota(G, \mathcal{F}) \le |I| + |D| = \iota(G[A]) + \gamma(G[B])$. Conversely, let $S$ be a minimum $\mathcal{F}$-isolating set of $G$ and let $A = V \setminus N[S]$ and $B = N[S]$. Then $G[A]$ is $\mathcal{F}$-free and hence $\iota(G[A], \mathcal{F}) = 0$. Moreover, $|S| \ge \gamma(G[B])$. Thus, $\iota(G, \mathcal{F}) = |S| \ge \iota(G[A], \mathcal{F}) + \gamma(G[B])$.\\
(iii) Let $H$ be an induced $\mathcal{F}$-free graph of $G$. Then, by item (ii), $\iota(G, \mathcal{F}) \le \iota(H, \mathcal{F}) + \gamma(G \setminus H) = \gamma(G \setminus H)$. On the other side, let $S$ be a minimum $\mathcal{F}$-isolating set of $G$ and let $H = G -N[S]$. Then $H$ is $\mathcal{F}$-free. Again by item(ii), it follows that $|S| = \iota(G, \mathcal{F}) \le \iota(H), \mathcal{F}) + \gamma(G \setminus H)$.\\
\end{proof}

\begin{lem}\label{la-dom-quot}
Let $G$ be a graph on $n$ vertices. The following assertions hold.
\begin{enumerate}
\item[(i)] For a graph $H$, $\iota_{H}(G) \le \gamma(H) \left\lfloor \frac{n}{n(H)} \right\rfloor$ and this is sharp if $\gamma(H) = 1$.
\item[(ii)] For a family of graphs $\mathcal{F}$, $\iota(G, \mathcal{F}) \le \displaystyle \sup_{F \in \mathcal{F}}\; \frac{\gamma(F)}{n(F)} n$.
\end{enumerate}
\end{lem}

\begin{proof}
(i) We will prove by induction on $n$ that $\iota_H(G) \le \gamma(H) \lfloor \frac{n}{n(H)} \rfloor$. If $1 \le n \le n(H)$, then $G$ is either $H$-free or  $H$ is a spanning subgraph of $G$, i.e. $n = n(H)$. In the first case we have $\iota_H(G) = 0$ and the statement follows trivially. In the second, using Lemma \ref{la-dom-sum} (i), we have $\iota_H(G) \le \gamma(G) \le \gamma(H) = \gamma(H) \lfloor \frac{n}{n(H)} \rfloor$ and we are done. Hence, assume that $G$ has order $n \ge n(H)+1$ and suppose the statement holds for any graph with less than $n$ vertices. If $G$ is $H$-free, then clearly $\iota_H(G) = 0$ and we are done. Otherwise let us consider a copy $H^*$ of $H$ in $G$ and let $G' = G \setminus H^*$. Then, by the induction hypothesis, $\iota_H(G') \le \gamma(H) \lfloor\frac{n(G')}{n(H)}\rfloor$. Let $S$ be a minimum $H$-isolating set of $G'$ and let $D$ be a minimum dominating set of $H^*$. Since $G- (N_G[S \cup D])$ is a subgraph of $G'-N_{G'}[S]$ and the latter is $H$-free, $S \cup D$ is an $H$-isolating set of $G$. Hence, 
\[\iota_H(G) \le |S \cup D| = |S| + \gamma(H^*) = |S| + \gamma(H) \le  \gamma(H) \left\lfloor\frac{n(G')}{n(H)}\right\rfloor + \gamma(H) = \gamma(H) \left\lfloor\frac{n}{n(H)}\right\rfloor.\]
For the sharpness, consider the graph $G = t H$, where $H$ a nontrivial graph with $\gamma(H) = 1$. Then $n = n(G) = t \cdot n(H)$ and, clearly, $\iota_H(G) = t = \lfloor\frac{n}{n(H)} \rfloor = \gamma(H) \lfloor\frac{n}{n(H)} \rfloor$.\\
(ii) Let $q(\mathcal{F}) = \displaystyle \sup_{F \in \mathcal{F}} \; \frac{\gamma(F)}{n(F)}$. We will prove the statement by induction on $n$. If $G$ is $\mathcal{F}$-free, then we have $\iota(G, \mathcal{F}) = 0 \le q(\mathcal{F})n$. If $F \subseteq G$ and $n(F) = n$ for some $F \in \mathcal{F}$, then let $D$ be a minimum dominating set of $F$. Clearly, $D$ is a dominating set of $G$ and thus, using Lemma \ref{la-dom-sum}~(i), $\iota(G, \mathcal{F}) \le \gamma(G) \le \gamma(F) \le q(\mathcal{F}) n(F) = q(\mathcal{F}) n$. This covers the cases where $n \le \min\{n(F) \;|\; F \in \mathcal{F}\}$. Suppose that $n > \min\{n(F) \;|\; F \in \mathcal{F}\}$ and that we have proved the statement for all graphs of order less than $n$. We can assume that $G$ is not $\mathcal{F}$-free for otherwise it is done. Also the case $F \subseteq G$ and $n(F) = n(G)$ for an $F \in \mathcal{F}$ works as above. Hence, we can assume that there is a graph $F \in \mathcal{F}$ contained in $G$ such that the graph $G'$ obtained after deleting the vertices of $F$ in $G$ is not empty. Let $S$ be a minimum $\mathcal{F}$-isolating set of $G'$ and $D$ a minimum dominating set of $F$. By the induction hypothesis, $|S| = \iota(G', \mathcal{F}) \le q(\mathcal{F}) n(G')$. Moreover, $|D| = \gamma(F) \le q(\mathcal{F}) n(F)$. Since $S \cup D$ is an $\mathcal{F}$-isolating set of $G$, it follows that
\[\iota(G,\mathcal{F}) \le |S| + |D| \le q(\mathcal{F}) n(G') + q(\mathcal{F}) n(F) = q(\mathcal{F}) n.\]
\end{proof}

\begin{cor}\label{cor-bounds} Let $G$ be a graph on $n$ vertices. Then the following statements hold.
\begin{enumerate} 
\item[(i)] Let $k \ge 0$ be an integer such that $\Delta(F) \ge k+1$ for all $F \in \mathcal{F}$. Then $\iota(G, \mathcal{F}) \le \iota_k(G)$.
\item[(ii)] If $\mathcal{F}$ is a family of non-empty graphs, then $\iota(G, \mathcal{F}) \le \iota(G)$.
\item[(iii)] $\iota_k(G) \le \frac{1}{k+2}n$ and this is sharp; 
\end{enumerate}
\end{cor}

\begin{proof}
(i) Since $\Delta(F) \ge k+1$, the graph $K_{1,k+1} \subseteq F$ for all $F \in \mathcal{F}$ and thus $\iota(G, \mathcal{F}) \le \iota_k(G)$ follows from \ref{la-family} (iii).\\
(ii) This is item (i) for $k=0$.\\
(iii) This follows directly by Lemma~\ref{la-dom-quot}~(i) using $H =  K_{1,k+1}$.
\end{proof}

\comment{
In Section \ref{sec:upper-bounds} we are going to give a substantial improvement of Corollary \ref{cor-bounds} (iii) when $k=0$ and $G$ is a connected graph different from the $C_5$-cycle or the $K_2$. Namely, we will show that $\iota(G) \le \frac{n}{3}$ if  (Theorem \ref{thm-n/3}). We would like to stress here that, in view of Corollary \ref{cor-bounds} (ii), this bound yields us $\iota(G, \mathcal{F}) \le i_1(G) \le \iota(G) \le \frac{n}{3}$ for a graph $G$ of order $n$ whose components are all different from the $C_5$-cycle or the $K_2$ and where $\mathcal{F}$ is the family of all cycles. However, $\iota_1(K_2) = 0 < \frac{n(K_2)}{3}$ and $\iota_1(C_5) = 1 < \frac{n(C_5)}{3}$ and thus $\iota(G, \mathcal{F}) \le i_1(G) \le \frac{n}{3}$ holds always.
}



Given a graph $G$ with vertex set $V = \{v_1, v_2, \ldots, v_n\}$, we define a bipartite graph $B(G)$ the following way. Let $V_1 = \{v^1_1, v^1_2, \ldots, v^1_n\}$ and $V_2 = \{v^2_1, v^2_2, \ldots, v^2_n\}$ be the two partite sets of $B(G)$ and, for $i, j \in \{1,2, \ldots, n\}$, let $v^1_i$ be adjacent to $v^2_j$ if and only if either $i = j$ or $v_i$ is adjacent to $v_j$ in $G$. Note that $|B(G)| = 2 n$ and $\delta(B(G)) = \delta(G) + 1$.

The following (technical) theorem will be used several times in the sequel.

\begin{thm}\label{thm-dom}
Let $G$ be a graph of order $n$. Let $\mathcal{F}$ be a family of graphs with $r = \min \{n(F) \;|\; F \in \mathcal{F}\}$. Then
\begin{enumerate}
\item[(i)] $\gamma(G) \le \iota(G \times K_r, \mathcal{F}) \le \min\{n, (r-1) \gamma(G) + \iota(G,\mathcal{F}), r \iota(G,\mathcal{F}) + \alpha(G, \mathcal{F}) \} \le \min \{n,r \gamma(G)\}$
\item[(ii)] $\gamma(G) \le \iota(B(G)) \le \gamma(B(G)) \le 2 \gamma(G)$.
\end{enumerate}
\end{thm}

\begin{proof}
(i) Let $V(G) = \{x_1, x_2, \ldots, x_n\}$ and let $H = G \times K_r$ be given by the vertex set $V(H) = V_1 \cup V_2 \cup \ldots \cup V_r$ with $V_i = \{x^i_1, x^i_2, \ldots, x^i_n\}$ where $H[V_i] \cong G$ for $i = 1, 2, \ldots, r$ and such that $x^i_k$ is adjacent to $x^i_k$ for any $1 \le i < j \le r$ and $1 \le k \le n$. For the first inequality, let $S$ be a minimum $\mathcal{F}$-isolating set of $H$. Let $Q$ be the set of vertices which are not dominated by $S$ in $H$ and let $Q_i = V_i \cap Q$ and $S_i = V_i \cap S$, $1\le i \le r$. We will show that the set $D =  \bigcup_{i=1}^{r} \{x^1_k \;|\; x^i_k \in S_i\}$ is a dominating set of $H[V_1] \cong G$.  Fix an integer $k \in \{1,2,\ldots, r\}$. If $x^1_k \in V_1 \setminus Q_1$ then either $x^1_k \in S_1$ or $x^1_k$ is dominated by a vertex in $S_1$ or $x^1_k$ is dominated by a vertex $x^j_k \in S_j$ for some $j \in \{2,3,\ldots,r\}$. The first and third cases imply that $x^1_k \in D$ while from the second case follows that $x^1_k$ is dominated by a vertex in $D$. Hence suppose that $x^1_k \in Q_1$. Note that, for any $\ell \in \{1,2,\ldots, n\}$, $H[\{x_{\ell}^i \;|\; 1 \le i \le r\}] \cong K_{r}$ is not $\mathcal{F}$-free. Thus, since $H[Q]$ is $\mathcal{F}$-free, there has to be a $j \in \{2,3,\ldots,r\}$ such that $x^j_k \notin Q_j$. Hence, either $x^j_k \in S_j$, or there is a vertex $x^j_{k'} \in S_j$ that dominates $x^j_k$, or $x^j_k$ is dominated by a vertex $x^{j'}_k \in S_{j'}$ for some $j' \in \{1,2,\ldots,r\}$, $j \neq j'$. This implies that either $x^1_k \in D$ or $x^1_{k'}$ dominates $x^1_k$. Hence, $D$ is a dominating set of $H[V_1] \cong G$ and thus $\gamma(G) \le |D| \le \sum_{i=1}^{r} |S_i| = |S| = \iota(H, \mathcal{F})$.\\
For the second inequality, note first that $V_1$ is a dominating set of $H$ and hence also an $\mathcal{F}$-isolating set of $H$. Thus $\iota(H, \mathcal{F}) \le n$. Now let $D$ and $S$ be, respectively, a minimum dominating set and a minimum $\mathcal{F}$-isolating set of $G$ and define $D_i = \{x^i_k \;|\; x_k \in D\}$ and $S_i = \{x^i_k \;|\; x_k \in S\}$, for $1 \le i \le r$. Then both $\cup_{i=1}^{r-1} D_i \cup S_r$ and $\cup_{i=1}^{r-1} S_i \cup (V_r \setminus S_r)$ are $\mathcal{F}$-isolating sets of $H$. Now the properties $|D_i| = \gamma(G)$, $|S_i| = \iota(G, \mathcal{F})$ and $|V_r \setminus S_r| \le \alpha(G, \mathcal{F})$ yield the desired inequality.\\
The last inequality is due to Lemma~\ref{la-dom-sum}~(i).
 \\
(ii) For the first inequality, let $S$ be a minimum isolating set of $B(G)$. Let $D = \{ v_k \;|\; v^i_k \in S \mbox{ for some $i$}\}$. Let $Q = V(B(G)) \setminus N_{B(G)}[S]$ be the set of isolated vertices in $B(G)$. Let $v_l$ be a vertex in $V(G) \setminus D$. Then $v^i_l \notin S$ for $i = 1, 2$. Since $v_l^1$ is adjacent to  $v_l^2$, one of both vertices, say $v_l^1$, is not in $Q$. This implies that $v_l^1 \in N_{B(G)}(S)$. Hence, there is a neighbor of $v_l^1$ in $S$ and thus there is a vertex in $D$ which is adjacent to $v_l$ in $G$. It follows that $D$ is a dominating set of $G$, which gives us $\gamma(G) \le |D| \le |S| = \iota(B(G))$.\\
The second inequality follows from Lemma \ref{la-dom-sum}~(i). For the last inequality, let $D$ be a minimum dominating set of $G$. Then $D' = \{v^1_l , v^2_l \;|\; v_l \in D\}$ is a dominating set of $B(G)$ and thus $2 \gamma(G) = 2|D| = |D'| \ge \gamma(B(G))$.
\end{proof}

\section{Upper bounds}\label{sec:upper-bounds}

In this section, we are going to present some upper bounds on $\iota(G)$ and $\iota_k(G)$ in terms of the order and the maximum and minimum degree of $G$. We will make emphasis on weather the graph is connected or not.

\subsection{Bounds in terms of order: connected graphs}

The following theorem, which is mentioned in the abstract, is one of the main results of this paper and a prototype of problems we pose in Section \ref{problems}.\\
 
\begin{thm}\label{thm-n/3}
Let $G$ be a connected graph on $n \ge 3$ vertices and different from $C_5$. Then $\iota(G) \le \frac{n}{3}$ and this bound is sharp.
\end{thm}

\begin{proof}
We will prove the satement by induction on $n$. Let $G$ be a connected graph on $n \ge 3$ vertices and different from the cycle $C_5$. If $n = 3$, evidently $\iota(G) = 1 = \frac{n}{3}$. So suppose $n > 3$ and assume the statement is valid for all graphs different from $C_5$ and with less than $n$ vertices. Select a vertex $r \in V(G)$ and let $T$ be a BFS-tree of $G$ rooted in $r$. Denote by $L$ the set of leaves of $T$. Let $V_0 = \{r\}$ and let $V_i$ denote the set of vertices of the $i$-th generation after $r$. Let $\ell$ be the last generation of vertices in $T$. So, $V = V_0 \cup V_1 \cup V_2 \cup \ldots \cup V_{\ell}$ and, since $n > 3$, $\ell \ge 1$. 

If $\ell = 1$, then $\{r\}$ is an isolating set of $G$ and the statement holds trivially. Hence, we may assume that $\ell \ge 2$.

Suppose now that, for some $k \le \ell - 1$, there is a vertex $u \in V_k$ such that $N_T(u) \cap V_{k+1} \subseteq L$ and $|N_T(u) \cap V_{k+1}| \ge 2$. Let $U = \{u\} \cup (N_T(u) \cap V_{k+1})$ and $G^* = G - U$. If $U = V(G)$, then clearly $u$ is an isolating set of $G$ and $\iota(G) = 1 \le \frac{n}{3}$ . Hence suppose $V(G) \setminus U  \neq \emptyset$.  Then $G^*$ is connected. If $n(G^*) \le 2$, then $\{u\}$ is an isolating set of $G$ and $\iota(G) = 1 \le \frac{n}{3}$ holds. If $G^* \cong C_5$, let $G^* = x_1x_2x_3x_4x_5x_1$ being $x_1$ the father of $u$ in $T$. Then $\{u, x_3\}$ is an isolating set of $G$ and, as $n \ge 8$, $i(G) \le 2 \le \frac{n}{3}$ is fulfilled. If $G^*$ is different from $C_5$ and $n(G^*) \ge 3$, then, by the induction hypothesis, there is an isolating set $S^*$ of $G^*$ with $|S^*| \le \frac{n-|U|}{3}$. Note that $S^* \cup \{u\}$ is an isolating set of $G$. Hence, since $|U| \ge 3$, we obtain $\iota(G) \le  \frac{n-|U|}{3}+ 1 \le  \frac{n}{3}$ and we are done. So we can assume in the following that, for any $k \le \ell-1$, all vertices $u \in V_k \setminus L$ fulfill either\\
(i) $u$ has only one child $v$ and $v \in L$ or \\
(ii) $u$ has grandchildren. 

Since $\ell \ge 2$, there exists a vertex $u \in V_{\ell-2}$ such that it has grandchildren. Consider now the following cases.

\noindent
{\it Case 1: $u$ has only one child $v$ in $T$.} Since $v \in V_{\ell - 1}$, $v$ has no grandchildren. Hence, by the assumption above, $v$ has only one child, say $w$, which is a leaf in $T$. Let $G^* = G - \{u,v,w\}$.  Then $G^*$ is nonempty and connected. 
We distinguish the following subcases.\\
{\it Subcase 1.1: $n(G^*) \le 2$.} Then $n \le 5$ and it is streightforward to check that, with exception of $G \cong C_5$, $\iota(G) \le \frac{n}{3}$. \\
{\it Subcase 1.2: $G^* \cong C_5$.} Then $n = 8$. Let $G^* = x_1x_2x_3x_4x_5x_1$ being $x_1$ the father of $u$ in $T$. If $x_5 \notin N_G(w)$, then $\{u, x_3\}$ is an isolating set of $G$. Similarly, if $x_2 \notin N_G(w)$, then $\{u, x_4\}$ is an isolating set of $G$. If, on the other side, $x_2$ and $x_5$ are both neighbors of $w$ in $G$, then, depending if $x_4 \notin N_G(u)$ or $x_4 \in N_G(u)$, either $\{w, x_2\}$ or $\{u, x_2\}$ is an isolating set of $G$. Hence, in all cases we obtain an isolating set with $2$ vertices and so $\iota(G) \le 2 \le \frac{n}{3}$.\\
{\it Subcase 1.3: $n(G^*) \ge 3$ and $G^* \neq C_5$.} Then, by induction, there is an isolating set $S^*$ of $G^*$ with $|S^*| \le \frac{n - 3}{3}$. Since $S^* \cup \{v\}$ is an isolating set of $G$, we obtain easily $\iota(G) \le \frac{n - 3}{3} + 1 =  \frac{n}{3} $.

\noindent
{\it Case 2: $u$ has at least two children in $T$.} Let $A$ the set of children of $u$ and $B$ the set of grandchildren of $u$ in $T$ and define $U = \{u\} \cup A \cup B$. Let $B_1 = \{ x \in B \; | \; N_G(x) \cap (V \setminus U) = \emptyset\}$ and $B_2 = B \setminus B_1$. Finally, set $U' = U \setminus B_2$ and let $I$ be the set of isolated vertices in $G[B_1]$.  Now we have two subcases.\\

\noindent
{\it Subcase 2.1: $B_1 \setminus I = \emptyset$.} Then $\{u\}$ is an isolating set of $G[U']$. Now let $G^* = G - U'$. Clearly, $G^*$ is either empty or connected. Namely, if $B_2 = \emptyset$, then $G*$ consists of the vertices from $T - U'$ which is a tree with a leaf in the father of $u$. On the other side, if $B_2 \neq \emptyset$, then by definition all its vertices are adjacent to some vertex of the tree $T - U$, which is also a tree with a leaf in the father of $u$.  Again, we consider here three different subcases.\\
{\it (i) $n(G^*) \le 2$.} If $n(G^*) = 0$, then $B_2 = \emptyset$ and $\{u\}$ is a separating set of $G$. If $n(G^*) = 1$, then, as the vertices of $B_2$ have all a neighbor in $V(G) - U$, it forces again $B_2 = \emptyset$. Hence, again, $\{u\}$ is an isolating set of $G$. Thus, in both cases we have $\iota(G) = 1 \le \frac{n}{3}$. If $n(G^*) = 2$ let $V(G^*) = \{x, y\}$, being $x$ the father of $u$ in $T$. If $N_G(y) \cap B_1 = \emptyset$, then $B_1 \cup \{y\}$ is an independent set and, since $u$ dominates $A \cup \{x\}$, $\{u\}$ is an isolating set of $G$ and thus $\iota(G) = 1 \le \frac{n}{3}$. Finally, suppose that $N_G(y) \cap B_1 \neq \emptyset$. Then, $|B_1| \ge 1$ and, with $|A| \ge 2$ and $n(G^*) \ge 2$, we have $n \ge 6$. As clearly $\{u, x\}$ is an isolating set of $G$, it follows again $\iota(G) \le 2 \le \frac{n}{3}$. \\
{\it (ii) $G^* \cong C_5$.} Again, let $G^* = x_1x_2x_3x_4x_5x_1$ being $x_1$ the father of $u$ in $T$. Since the vertices in $B_1 = I$ form an independent set and have no neighbors in $G^*$, it is not difficult to see that $\{u, x_3\}$ is an isolating set of $G$ (the set $A \cup \{x_1, x_2, x_4\}$ is dominated and te rest $B_1 \cup \{x_5\}$ is independent). Hence, as $n \ge 8$, $\iota(G) \le 2 \le \frac{n}{3}$.\\
{\it (iii) $n(G^*) \ge 3$ and $G^* \neq C_5$.} Let $S^*$ be a minimum isolating set of $G^*$. By the induction hypothesis, $|S^*| \le \frac{n - |U'|}{3}$. Moreover, since $N_G(B_1) \cap V(G^*) = \emptyset$ and $B_1$ is an independent set in $G$, $S^* \cup \{u\}$ is an isolating set of $G$. Hence, $\iota(G) \le |S^* \cup \{u\}| \le \frac{n - |U'|}{3} + 1 \le \frac{n}{3} $ and we are done. 

\noindent
{\it Subcase 2.2: $B_1 \setminus I \neq \emptyset$.} Then $\delta(G[B_1 \setminus I]) \ge 1$ and $|B_1 \setminus I| \ge 2$. Let $y , z\in B_1$ be two adjacent vertices in $G$ and let $x \in A$ be the father of $y$ in $T$. Define $G^* = G - \{x,y,z\}$. Note that, by assumption and since $x \in V_{\ell-1}$, $x$ can have only one child, which is $y$. This implies that $G$ is connected and $n(G^*) \ge 2$. For the last, we distinguish the following three subcases.\\
{\it (i) $n(G^*)=2$.} Since $u$ has at least two children in $T$ and $n(G^*) =2$, $|A| = 2$, say, $A = \{x,v\}$. Then, clearly, $G^* = uv$. As $z$ has to have a father in $A$ different from $x$, it follows that $G$ is either the $5$-cycle $C = uxyzvu$ or the graph $C+vx$. Since by hypothesis $G \neq C$, it follows that $G = C+vx$. Thus, $\{x\}$ is an isolating set and $\iota(G) = 1 \le \frac{n}{3}$. \\
{\it (ii) $G^* \cong C_5$.} Then $n = 8$. Let $G^* = uu_1u_2u_3u_4u$ and assume, without loss of generality that $u_1 \in A$ is the father of $z$ in $T$. Recall that, in $G$, none of the vertices of $V(G) \setminus U$ is adjacent to the vertices in $B_1$. In particular, $u_3$ and $y$ are not adjacent. Hence, it follows that $\{u, u_1\}$ is an isolating set of $G$, implying that $\iota(G) \le 2 \le \frac{n}{3}$. \\
{\it (iii) $n(G^*) \ge 3$ and $G \neq C_5$.} Then, by the induction hypothesis, there is an isolating set $S^*$ of $G^*$ with $|S^*| \le \frac{n-3}{3} $. Since $S^* \cup \{y\}$ is an isolating set of $G$, we obtain again
$\iota(G) \le |S^*| + 1 \le \frac{n-3}{3} + 1 = \frac{n}{3}$ and we are done.

For the sharpness, consider the following graphs.\\
\noindent
(i) Let $G_1$ be the graph consisting of an arbitrary connected graph $H_1$ on $n_1$ vertices, each of which is attached to a $K_2$ by means of an edge. Then $\delta(G_1) = 1$ and $\iota(G_1) = n_1 = \frac{n(G_1)}{3}$.\\
(ii) Let $G_2$ be the graph consisting of an arbitrary connected graph $H_2$ on $n_2$ vertices, each of which is attached to a $K_2$ by means of two edges (these edges going each to a different vertex of the $K_2$). Then $\delta(G_2) = 2$ and $\iota(G_2) = n_2 = \frac{n(G_2)}{3}$.
\end{proof}

\begin{cor}\label{cor-n/3-1}
Let $G$ be a graph on $n$ vertices with no component on less than $3$ vertices. Then the following holds.
\begin{enumerate}
\item[(i)] $\iota(G) \le \frac{2n}{5}$ and equality holds if and only if $G$ is the union of vertex disjoint copies of $C_5$.
\item[(ii)] Let $G_1, G_2, \ldots, G_k$ be the components of $G$. If $G_i \neq C_5$ for all $1 \le i \le k$, then $\iota(G) \le \displaystyle \sum_{1 \le i \le k} \left \lfloor \frac{n(G_i)}{3}\right \rfloor \le \left \lfloor \frac{n}{3}\right \rfloor$.
\end{enumerate}
\end{cor}

\subsection{Bounds in terms of order and maximum degree}

\begin{thm}\label{thm-S}
Let $G$ be a graph of order $n$ and with vertex set $V$. Consider a subset $S \subseteq V$. Then the following statements hold.
\begin{enumerate}
\item[(i)] $\iota(G) \le  \frac{n - |N(S)| + |S|}{2}$.
\item[(ii)] If $\delta(G - N[S]) \ge 2$, then $\iota(G) \le \frac{2n - 2|N(S)| + 3|S|}{5}$.
\item[(iii)] If every component of  $G - N[S]$ has at least $3$ vertices and no $C_5$-component,  then $ \iota(G) \le  \frac{n - |N(S)| +2|S|}{3}$.  
\end{enumerate}
\end{thm}

\begin{proof}
(i) This follows because of $\iota(G) \le \iota(G -N[S]) +|S| \le  \frac{n - |N[S]|}{2} + |S| =  \frac{n - |N(S)| + |S|}{2}$.\\
(ii) If $\delta(G - N[S]) \ge 2$, then there is no component of order smaller than $3$ and, hence, by Corollary \ref{cor-n/3-1}(i), we infer that $\iota(G) \le \iota(G -N[S]) +|S| \le  \frac{2(n - |N[S])|}{5} + |S| =  \frac{2n - 2|N(S)| + 3|S|}{5}$.\\
(iii) Since  every component of  $G - N[S]$ has at least $3$ vertices and no $C_5$-component, it follows by Theorem \ref{thm-n/3} $\iota(G) \le \iota(G -N[S]) +|S| \le \frac{n- |N[S]|}{3} + |S| =  \frac{n - |N(S)| +2|S|}{3}$.
\end{proof}

\begin{cor}\label{cor-n/3-2}
Let $G$ be a graph on $n$ vertices and maximum degree $\Delta$. 
\begin{enumerate}
\item[(i)] Then $\iota(G) \le \frac{n - \Delta + 1}{2}$ and this is sharp for various values of $\Delta \ge \frac{n}{2}$.
 \item[(ii)] If there is a vertex $v$ of maximum degree $\Delta$ and all components of $G - N[v]$ contain at least $3$ vertices and no $C_5$-component,  then $\iota(G)\le \frac{n -\Delta +2}{3}$.  
 \item[(iii)] If $G$ is a graph of maximum degree $\Delta$, then $\iota_1(G) \le \frac{n-\Delta+2}{3}$ and this is sharp.
\end{enumerate}
\end{cor}

\begin{proof}
(i) Let $x$ be a vertex of maximum degree $\Delta$. Let $G' = G - N_G[x]$ and let $I$ the set of isolated vertices in $G'$. Then there is a dominating set $D$ of $G'-I$ with $|D| \le \frac{n(G')-|I|}{2} \le \frac{n(G')}{2} = \frac{n - \Delta -1}{2}$. Since $D \cup \{x\}$ is an isolating set of $G$, we obtain $\iota(G) \le |D|+1 \le \frac{n - \Delta +1}{2}$. The sharpness can be seen with the following construction. For two positive integers $p$ and $q$ let $x$ be the center of a star $K_{1,2p+q}$, from which exactly $2p$ edges are subdivided. Now join by pairs the leaves of the subdivided edges with an edge. The graph $G$ obtained this way has $n(G) = n$ and $\iota(G) = p +1 = \frac{n- (2p+q)+1}{2}$.\\
(ii) This follows directly from Theorem \ref{thm-S}~(iii) with $S = \{v\}$.\\
(iii) Let $v$ be a vertex of maximum degree $\Delta$. Consider the graph $G^* = G - N[v]$. Let $A \subseteq V(G^*)$ be the set of vertices of all components of at most $2$ vertices in $G^*$. Let $B$ be the set of vertices of all $C_5$-components of $G^*$. Let $I$ be a minimum $\{K_2\}$-isolating set of $G[B]$. Then clearly $|I| = \frac{|B|}{5}$. Finally, let $J$ be a minimum isolating set of $G^* - (A \cup B)$ (if $V(G^*) \setminus (A \cup B) = \emptyset$, set $J = \emptyset$). Note that $\{v\} \cup I \cup J$ is a $\{K_2\}$-isolating set of $G$ and,  by Theorem \ref{thm-n/3}, $|J| \le \frac{n(G^*) - |A| - |B|}{3}$. Hence, we obtain 
\begin{align*}
\iota_1(G) &\;\le\; 1 + |I| + |J| 
                  \;\le\; 1 + \frac{|B|}{5} + \frac{n(G^*) - |A| - |B|}{3} \\
                  &\;\le\; 1 + \frac{|B|}{3} + \frac{n(G^*) - |B|}{3} 
                  \;=\; 1+ \frac{|V(G) - N[v]|}{3}\; = \;\frac{n - \Delta +2}{3}.
\end{align*}
Let $r, s, t$ be non-negative integers with $r+s+t \ge 1$. Then the graph $F_{r,s,t} = rK_3 \cup sP_3 \cup tC_6$ has $\iota_1(F_{r,s,t}) = r + s + 2t = \frac{n(F_{r,s,t})}{3}$ and thus the bound is sharp.
\end{proof}

\subsection{Bounds in terms of order and minimum degree}

Denote by $f(\delta, k) = \inf \{\alpha \;|\; \iota_k(G) \le \alpha |G| \mbox{ for every graph } $G$ \mbox{ with minimum degree } \delta \}$. In case $k = 0$ we shall use the notation $f(\delta)$.

 

\begin{thm}\label{thm:up-bound-min-deg}
The following statements hold.
\begin{enumerate} 
\item[(i)] For $\delta \ge 1$, $\frac{2}{\delta+3} \le f(\delta) \le \frac{\ln(\delta+1)+ \frac{1}{2}}{\delta+1}$.
\item[(ii)] $f(1) = \frac{1}{2}$, $f(2) = \frac{2}{5}$ and $f(3) = \frac{1}{3}$ and this is sharp.
\item[(iii)] 
For $\delta \ge k+1$,
\[
(1 - o(1)) \frac{\ln(\delta+1)}{(k+2) (\delta+1)} \le f(\delta,k) \le \frac{\ln(\delta+\frac{1}{2}) + 1}{\delta+1}.
\]
\end{enumerate}
\end{thm}

\begin{proof}
(i) We first prove the lower bound. Here for, consider the following graph. Let $n$ be divisible by an integer $r \ge 4$ and let $H$ be a complete graph $K_r$ of order $r$ to which the edges of a Hamiltonian cycle are deleted. Let $G$ be the graph consisting of $\frac{n}{r}$ copies of $H$. Then $\delta = \delta(G) = r-3$ and each vertex dominates all vertices on the copy of $H$ it belongs with exception of two adjacent vertices, showing that $\iota(G) = 2\frac{n}{r}  = \frac{2}{\delta+3}n$.\\
For the upper bound, we follow the proof for the probabilistic upper bound for the domination number due to Alon (see \cite{AlSp}) but, instead of including all non-dominated vertices, we only need to take at most the half of those which are not isolated. So let $G$ be a graph with minimum degree $\delta \ge 1$. Let $p \in [0,1]$. Select a set of vertices $A$ independently at random such that $P(v \in A) = p$. Let $I$ be the set of isolated vertices in $V \setminus A$ and let $B = V \setminus (N[A] \cup I)$. Since there are no isolated vertices in $B$, we know by Ore \cite{Ore}, that there is a dominating set $D$ of $G[B]$ such that $|D| \le \frac{|B|}{2}$. Then, clearly, $A \cup D$ is an isolating set of $G$. Note that $E[|D|] \le E[\frac{|B|}{2}] = \frac{1}{2} E[|B|].$
Hence, since \[P(v \in B) = P(v \in V\setminus N[A]) = (1-p)^{\deg(v)+1} \le (1-p)^{\delta+1},\]
we obtain, using $1-x \le e^{-x}$ for $x \ge 0$,
\[E[|A \cup D|] 
         \le E[|A|] + \frac{1}{2}E[|B|] = pn + \frac{1}{2} (1-p)^{\delta+1}n
         \le \left( p + \frac{1}{2}e^{-p(\delta+1)}\right)n.\]
Since the function $f(x) = e^{-x(\delta+1)}$ attains its minimum when $x = \frac{\ln(\delta+1)}{\delta+1}$, we can take $p = \frac{\ln(\delta+1)}{\delta+1}$ in order to obtain the minimum value in the above inequality chain. Hence, the expected cardinality for an isolating set is at most
\[\frac{\ln(\delta+1)+\frac{1}{2}}{\delta+1}\;n,\]
which gives the desired upper bound for $\iota(G)$.\\
(ii) The inequalities $f(1) \ge \frac{1}{2}$, $f(2) \ge \frac{2}{5}$ and $f(3) \ge \frac{1}{3}$ follow from item (i). Let now $G$ be a graph of order $n$. If $G$ has no isolated vertices, then $\gamma(G) \le \frac{n}{2}$ \cite{Ore}, and hence by Lemma \ref{la-dom-sum}~(i) $\iota(G) \le \frac{n}{2}$. If $\delta(G) \ge 2$, it is well known that $\gamma(G) \le \frac{2}{5}n$ unless $G$ belongs to a family of $7$ exceptional graphs ($P_4$ and six graphs of order $7$) \cite{McCShe}. For $P_4$ we have $\iota(P_4)=1$ and for the other $6$ exceptional graphs on $7$ vertices we checked that there is an isolating set on $2 = \frac{2}{7}n$ vertices. Hence, using Lemma \ref{la-dom-sum}~(i) for all other graphs $G$ with minimum degree $\delta \ge 2$, we obtain in all cases $\iota(G) \le \frac{2}{5}n$. Finally, if $G$ has minimum degree $\delta \ge 3$, then all components of $G$ have at least $4$ vertices and none of them is a $C_5$ and hence Theorem \ref{thm-n/3} yields $\iota(G) \le \frac{1}{3}n$.\\
(iii) For the lower bound, consider a graph $G$ of minimum degree $\delta(G) = \delta-1 \ge 0$ such that $\gamma(G) = (1 - o(1)) \frac{\ln \delta}{\delta} n(G)$, whose existence was given by Alon and Wormald in \cite{AlWo}.  Now let $H = G \times K_{k+2}$. Then $n(H) = (k+2) n(G)$ and $\delta(H) = \delta(G) + k+1 \ge k+1$. By Theorem \ref{thm-dom}~(i), it follows
\begin{align*}
\iota_k(H) \ge \gamma(G) &= (1 - o(1)) \frac{\ln \delta}{\delta} n(G)\\
                &= (1 - o(1)) \frac{(\delta+1)\ln \delta}{ \delta \ln(\delta+1)} \cdot \frac{\ln(\delta+1)}{(k+2) (\delta+1)} \cdot (k+2) n(G)\\
                &\ge (1 - o(1))\frac{\ln(\delta+1)}{(k+2) (\delta+1)} n(H),
\end{align*}
and we are done.
For the upper bound, let $G$ be any graph on $n$ vertices and with minimum degree $\delta \ge k+1$. By Lemma \ref{la-dom-sum}~(i), we have $\iota_k(G) \le \gamma(G)$. Hence, by the bound on domination due to Aranutov, Lov\'asz and Payan \cite{Arn,Lov,Pay} (for which Alon gave the probabilistic proof cited above in the proof of item (iii)), $\iota_k(G) \le \gamma(G) \le \frac{\ln(\delta+1)+1}{\delta+1} n$.
\end{proof}

Note that the lower bounds in items (i) and (iii) of Theorem \ref{thm:up-bound-min-deg} are both in order as the first gives a better (and explicit) lower bound for small values of $\delta$.

\begin{thm}
Let $G$ be a bipartite graph on $n$ vertices and minimum degree $\delta$. Then $\iota(G) \le \frac{\ln \delta+1}{2\delta}n$ and this is nearly sharp.
\end{thm}
\begin{proof}
Let $G$ be bipartite with bipartition $V_1 \cup V_2$. Let $|V_i| = n_1$ and $|V_2| = n_2$ and assume that $n_1 \le n_2$. Choose a subset $A \subseteq V_1$ each vertex of it independently and at random with probability $P(v \in A) = p$. Let $B \subseteq V_2$ be the set of vertices in $V_2$ having no neighbor in $A$. Then $A \cup B$ is an isolating set of $G$ and thus $\iota(G)$ is at most as large as the expected size of $|A \cup B|$. Note that
\begin{eqnarray*}
E[|A \cup B|] = E[|A|] + E[|B|] = n_1 p + \sum_{v \in V_2} (1-p)^{\deg(v)} \le n_1 p + n_2 (1-p)^{\delta}.
\end{eqnarray*} 
Considering the function $f(x) = n_1 x + n_2 (1-x)^{\delta}$ and its derivative $f'(x) = n_1 - n_2 \delta (1-x)^{\delta-1}$, we can see that $f'(x) = 0$ when $(1-x)^{\delta-1} = \frac{n_1}{n_2 \delta}$. Since $n_1 \le n_2$, we obtain $(1-x)^{\delta-1} \le \frac{1}{\delta}$ and thus $x \ge 1- (\frac{1}{\delta})^{\frac{1}{\delta-1}}$. Hence, we can choose $p = 1- (\frac{1}{\delta})^{\frac{1}{\delta-1}}$. It follows that
\begin{eqnarray*}
\iota(G) \le E[|A \cup B|] \le  n_1 p + n_2 (1-p)^{\delta} = n_1 \left(1- \left(\frac{1}{\delta}\right)^{\frac{1}{\delta-1}}\right) + n_2 \frac{1}{\delta} \le n_1 \frac{\ln \delta}{\delta}+ n_2 \frac{1}{\delta}.
\end{eqnarray*}
Since $n _1 \le n_2$ and the worst case is when $n_1 = n_2 = \frac{n}{2}$, we infer that
$\iota(G) \le \frac{\ln \delta + 1}{2\delta}n$.\\
For the sharpness, consider an Alon-Wormald \cite{AlWo} $(\delta-1)$-regular graph $G$ on $n$ vertices having $\gamma(G) = (1 - o(1)) \frac{\ln \delta}{\delta}n$. Further, take the bipartite graph $B(G)$ described just before Theorem \ref{thm-dom} and note that $|B(G)| = 2n$, $\delta(B(G)) = \delta$. Now, with Theorem \ref{thm-dom}~(ii), we obtain 
\[\iota(B(G)) \ge \gamma(G) \ge (1-o(1)) \frac{\ln \delta}{\delta}n= (1-o(1)) \frac{\ln \delta}{2\delta}n(B(G)),\]
showing that the upper bound given above is nearly sharp.
\end{proof}

In the following theorem, we will sow that $f(\delta,k)$ is monotonically decreasing as $\delta$ grows. This implies in particular that, for any graph $G$ with minimum degree $\delta(G) \ge \delta$, $\iota_k(G) \le f(\delta, k) n(G)$.

\begin{thm}
For $\delta \ge 1$, $f(\delta,k) \ge f(\delta+1,k)$.
\end{thm}

\begin{proof}
Since $f(\delta+1, k) = \lim_{n \rightarrow \infty} f(\delta+1, k,n)$, there is a sequence of graphs $(H_i)_{i\ge1}$ with $\delta(H_i) = \delta+1$ such that $n_i = n(H_i)$ tends to infinity as $i$ grows and 
\[\lim_{i \rightarrow \infty} \frac{\iota_k(H_i)}{n_i} = f(\delta+1,k). \]
Note that the infimum is either a minimum but then it is obtained by arbitrarily many copies of the graph that realizes the minimum or there is a sequence as above with $\lim_{i \rightarrow \infty}\frac{\iota_k(H_i)}{n_i} = f(\delta+1,k)$. So, in any case, we can use such a sequence.
Consider now the graphs $G_i = H_i \cup K_{\delta+1}$. Then $\delta(G_i) = \delta$, $n_i(G_i) = n_i + \delta + 1$ and $\iota_k(G_i) \ge \iota_k(H_i)$ for $i \ge 1$ and we obtain
\[f(\delta,k) \ge \lim_{i \rightarrow \infty} \frac{\iota_k(G_i)}{n(G_i)} \ge \lim_{i \rightarrow \infty}\frac{\iota_k(H_i)}{n_i+\delta+1} = \lim_{i \rightarrow \infty}\frac{\iota_k(H_i)n_i}{n_i(n_i+\delta+1)}  = f(\delta+1,k)\]
and we are done.
\end{proof}

\subsection{Bounds in terms of order and minimum degree: connected graphs}

\begin{thm} \label{thm-gral}
Let $\mathcal{F}$ be a family of graphs and let $\delta \ge 1$. Let $H$ be a connected graph with $\iota(H, \mathcal{F})= q$. Then, for arbitrarily large $n$, if $\delta = 1$  and $\delta(H) = 1$ or if $\delta \ge 2$ and $\delta(H) \ge  \delta$, there is a connected graph $G$ on $n$ vertices such that $\delta(G) = \delta$ and \[\iota(G, \mathcal{F}) \ge \frac{q}{n(H)+1}n.\]
\end{thm}

\begin{proof}
Let $H$ be a connected graph with $\delta(H) \ge  \delta \ge 2$ or $\delta(H) = \delta = 1$ and let $\iota(H, \mathcal{F}) = q$. Define a graph $G$ according to the following cases. \\
\textit{Case 1: $\delta = 1$.}
Take $t \ge \delta+1$ copies $H_1, H_2, \ldots, H_t$ of $H$ and a connected graph $G^*$ on $t$ vertices $v_1, v_2, \ldots, v_t$. For $1 \le i \le t$, select one vertex $u_i$ from $V(H_i)$ such that in $V(H_i) \setminus \{u_i\}$ there is still one vertex of degree $1$ in $H_i$ and connect $u_i$ and $v_i$ by an edge. Then $n(G) = n = t (n(H)+1)$ and, clearly, $\delta(G) = 1$. \\
\textit{Case 2: $\delta \ge 2$.}
Take $t \ge \delta+1$ copies $H_1, H_2, \ldots, H_t$ of $H$ and a connected graph $G^*$ on $t$ vertices $v_1, v_2, \ldots, v_t$ and minimum degree $\delta(G^*) = \delta - 1 \ge 1$. For $1 \le i \le t$, select one vertex $u_i$ from $V(H_i) = V_i$ and connect $u_i$ and $v_i$ by an edge. Then $n(G) = n = t (n(H)+1)$ and, clearly, $\delta(G) = \delta$.\\
Now, in both cases, let $S$ be a minimum $\mathcal{F}$-isolating set of $G$ and let $S_i = S \cap V_i$, for $1 \le i \le t$. Fix one $i \in \{1,2,\ldots,t\}$. If $v_i \notin S$, then $H_i - N_{H_i}[S_i] \subseteq G - N_G[S]$ and thus $H_i - N_{H_i}[S_i]$ is $\mathcal{F}$-free. Hence, $S_i$ is an $\mathcal{F}$-isolating set of $H_i$, which yields $|S_i| \ge q$. On the other side, if$v_i \in S$, then we have $H_i - N_{H_i}[S_i \cup \{u_i\}] \subseteq G-N_G[S]$ and thus $H_i - N_{H_i}[S_i \cup \{u_i\}] $ is $\mathcal{F}$-free. This implies that $S_i \cup \{u_i\}$ is an $\mathcal{F}$- isolating set of $H_i$ and thus $|S_i \cup \{u_i\}| \ge r$, from which we deduce $|S_i| \ge q-1$. For both constructions, it follows that $|(V_i \cup \{v_i\}) \cap S| \ge k$ for $1 \le i \le t$, yielding $\iota(G, \mathcal{F}) = |S| \ge t q= \frac{q}{n(H)+1}n$.
\end{proof}

The following two corollaries follow partially from Theorem \ref{thm-gral}. We shall need the following notation. 
We define a parameter $f_c(\delta,k)$ the following way. Let 
\[f_c(\delta,k,n) = \inf \{\alpha \in \mathbb{R} : \iota(G,k) \le \alpha \; n(G), G \mbox{ connected graph } n(G) \ge n, \delta(G) = \delta\}.\]
Observe that $f_c(\delta,k,n) \le f_c(\delta,k,n+1) \le 1$. Hence, for fixed $\delta$ and $k$, $f_c(\delta,k,n)$ is a monotone non-decreasing sequence, which is bounded from above. So, we may define
\[f_c(\delta,k) = \lim_{n \rightarrow \infty} f_c(\delta,k,n).\]
Further, in case $k = 0$ we shall use the notation $f_c(\delta)$ for $f_c(\delta,0)$.

\begin{cor}\label{cor-f_c}
The following statements hold.
\begin{enumerate}
\item[(i)] $f_c(1) = f_c(2) = \frac{1}{3}$.
\item[(ii)] For $\delta \ge 3$, $\frac{2}{\delta+4} \le f_c(\delta) \le f(\delta) \le \frac{\ln(\delta+1)+ \frac{1}{2}}{\delta+1}$.
\item[(iii)] For $\delta \ge k+1$, $(1 - o(1)) \frac{\ln(\delta+1)}{(k+2) (\delta+1)} \le f_c(\delta, k) \le f(\delta,k) \le \frac{\ln(\delta+\frac{1}{2}) + 1}{\delta+1}$.
\end{enumerate}
\end{cor}

\begin{proof}
(i) Since the graphs yielding the sharpness in the bound of Theorem \ref{thm-n/3} are connected and arbitrarily large, this is clear.\\
(ii) Consider the graph $H = K_r - C_r$, a complete graph on $r \ge 4$ vertices without a Hamiltonian cycle. Then $n(H) = r = \delta(H) + 3$ and $\iota(G) = 2$ and, thus, Theorem~\ref{thm-gral} gives $\frac{2}{\delta+4} \le f_c(\delta)$. Since, clearly, $f_c(G) \le f(G)$, the last inequality follows from Theorem~\ref{thm:up-bound-min-deg}~(i).\\
(iii) This follows from the construction of the graph yielding the lower bound of Theorem~\ref{thm:up-bound-min-deg}~(iii), which is already a connected arbitrarily large graph.
\end{proof}

\begin{cor}\label{cor-conn-isol} The following statements hold.
\begin{enumerate}
\item[(i)] For $\mathcal{F} = \{C_k \; | \; k \ge 3\}$, there are arbitrarily large connected graphs $G$ for which $\iota(G, \mathcal{F}) \ge \frac{1}{4} n(G)$.
\item[(ii)] There are arbitrarily large connected graphs $G$ for which $\iota_{K_k}(G) \ge \frac{1}{k+1} n(G)$.
\item[(iii)] For $\mathcal{F}_k$, the family of all trees of order $k \ge 2$, there are arbitrarily large connected graphs $G$ for which $\iota(G, \mathcal{F}_k) \ge \frac{1}{k+1} n(G)$.\\
\end{enumerate}
\end{cor}

\begin{proof}
For (i), (ii) and (iii), apply Theorem \ref{thm-gral} taking, respectively, $H = C_3, K_{k}$ and $K_{1,k-1}$.
\end{proof}

Now we are going to show that $f_c(\delta, k)$ is monotonically decreasing as $\delta$ grows (where $\delta \ge 2$). This implies in particular that, for any connected graph $G$ with minimum degree $\delta(G) \ge \delta$, $\iota_k(G) \le f_c(\delta, k) n(G)$.

\begin{thm}
For $\delta \ge 2$, $f_c(\delta,k) \ge f_c(\delta+1,k)$.
\end{thm}

\begin{proof}
Since $\displaystyle f_c(\delta+1,k) = \lim_{n \rightarrow \infty} f_c(\delta+1,k,n)$, there is a sequence $(H_i)_{i \ge 1}$ of connected graphs $H_i$ with $\delta(H_i) = \delta+1$, $n_i = n(H_i)$, $q_i = \iota_k(H_i)$ and $\displaystyle \lim_{i \rightarrow \infty} n_i$ tending to infinity such that 
\[\lim_{i \rightarrow \infty} \frac{q_i}{n_i} = f_c(\delta+1, k).\]
By Theorem \ref{thm-gral}, there are connected graphs $G_i$ with $\delta(G_i) = \delta \ge 2$ with $\frac{\iota_k(G_i)}{n(G_i)} \ge \frac{q_i}{n_i+1}$. Then we have 
\[
 f_c(\delta, k) \ge \lim_{i \rightarrow \infty} \frac{\iota_k(G_i)}{n(G_i)} \ge \lim_{i \rightarrow \infty} \frac{q_i}{n_i+1} = \lim_{i \rightarrow \infty} \frac{q_i n_i}{n_i (n_i+1)} = \lim_{i \rightarrow \infty} \frac{q_i}{n_i} = f_c(\delta,k)  
\]
and we are done.
\end{proof}


\section{Lower bounds}

Let $G$ be a graph. With $G^2$ we denote the \emph{power-$2$ graph} of $G$, that is, the graph that arises from $G$ by adding all edges between vertices within distance $2$. Recall also that $\alpha_k(G)$ is the $k$-independence number of $G$.

\begin{thm}\label{thm-lb-power-alpha} Let $G$ be a graph with minimum degree $\delta$ and maximum degree $\Delta$.
\begin{enumerate}
\item[(i)]  If $\delta \ge k+1$, then $\iota_k(G) \ge \gamma(G^2)$.
\item[(ii)] If $\Delta \ge k+1$, $\iota_k(G) \ge \frac{n+1- \alpha_k(G)}{\Delta+1}$ and this is sharp.
\end{enumerate}
\end{thm}

\begin{proof}
(i) Let $S$ be a minimum $k$-isolating set of $G$. Then every vertex in $N(S)$ has a neighbor in $S$. Also, since $\delta \ge k+1$ and $G - N[S]$ is $K_{1,k+1}$-free, every vertex in $V(G) \setminus N[S]$ has a neighbor in $N(S)$. Hence, every vertex in $V(G) \setminus N[S]$ is within distance $2$ from a vertex of $S$ and thus $S$ is a dominating set of $G^2$, yielding $\gamma(G^2) \le \iota_k(G)$.\\
(ii) Let $S$ be a minimum $k$-isolating set of $G$. Since $\Delta \ge k+1$, $S \neq \emptyset$. Let $x \in S$. Then, Since $G - N[S]$ is $K_{1,k+1}$-free and $x$ has no neighbors in $V(G) \setminus N[S]$, $(V(G) \setminus N[S]) \cup \{x\}$ is a $k$-independent set of $G$ and, therefore, $\alpha_k(G) \ge n - |N[S]| + 1$. Since every vertex in $S$ has at most $\Delta$ neighbors in $N(S)$, we obtain 
\[\alpha_k(G) \ge n - |N[S]| + 1 = n - |S| - |N(S)| + 1 \ge n - |S| - \Delta |S| + 1 = n - (\Delta+1) |S| + 1.\]
This implies $\iota_k(G) = |S| \ge \frac{n+1 - \alpha_k(G)}{\Delta+1}$. For the sharpness, consider a graph $G$ consisting of $\Delta$ copies $G_1, G_2, \ldots, G_{\Delta}$ of $K_{k+1}$ and a vertex $x$. Now select a vertex $x_i \in V(G_i)$ and include the edges $xx_i$, $1 \le i \le \Delta$. Then $\{x\}$ is a minimum $k$-isolating set and $V(G) \setminus \{x_1, x_2, \dots, x_{\Delta}\}$ is a maximum $k$-independent set. Hence, $\iota_k(G) = 1 = \frac{\Delta(k+1)+2 - (\Delta k + 1)}{\Delta+1} = \frac{n(G) +1 - \alpha_k(G)}{\Delta +1}$.
\end{proof}

In the following theorem, we give a lower bound for the isolation number $\iota(G)$ of a graph $G$ in terms of its maximum degree and average degree. For subsets $A, B \subseteq V(G)$, we will use the notation $m(A, B)$ for the number of edges with one vertex in $A$ and one in $B$.

\begin{thm}\label{thm-lb-degrees}
Let $G$ be a graph on $n$ vertices with average degree ${\rm d}$ and maximum degree $\Delta$. Then 
\[\iota(G) \ge \frac{{\rm d} n}{2 \Delta^2}.\]
Moreover, equality holds if and only if, for an integer $t \ge 1$, $G$ is a bipartite graph with partition sets $A$ and $B$, $|A| \le |B|$, where $A = \{a_{i,j} \, | \, 1 \le i \le t, \; 1 \le j \le \Delta\}$ and $S = \{s_i \,| \, 1 \le i \le t\} \subseteq B$, such that the following holds: $N(s_i) = \{a_{i,j} \, | \, 1 \le j \le \Delta\}$, $|N(a) \cap B| = \Delta$ for all $a \in A$ and $|N(b) \cap A| \le \Delta$ for all $b \in B$. 
\end{thm}
\begin{proof}
Let $V = V(G)$ and let $S$ be a minimum isolating set of $G$. We will bound the number of edges from above. Clearly, $m(S, N[S]) \le \Delta |S|$. Moreover, since the vertices of $N(S)$ have all at least one neighbor in $S$, we have $m(N(S), V \setminus S) \le (\Delta-1) |N(S)| \le (\Delta-1) \Delta |S|$. Hence,
\[m(G) \le \Delta |S| + (\Delta -1) \Delta |S| = \Delta^2 |S| = \Delta^2 \iota(G).\]
Thus it follows $\iota(G) \ge \frac{m(G)}{\Delta^2} = \frac{{\rm d}n}{2 \Delta^2}$. 

Now suppose that we have a graph $G$ with average degree ${\rm d}$ and maximum degree $\Delta$ such that $\iota(G) = \frac{{\rm d} n}{2 \Delta^2}$. Then, for any minimum isolating set $S$ of $G$, we have equalities in all the above inequalities and thus $m(S, N[S]) = \Delta |S|$ and $m(N(S), V \setminus S) = (\Delta-1) |N(S)| = (\Delta-1) \Delta |S|$. Hence, both $N(S)$ and $S \cup (V\setminus N[S])$ are independent sets. Setting $A = N(S)$ and $B = S \cup (V\setminus N[S])$, $S = \{s_i \,| \, 1 \le i \le t\}$, $N(s_i) = \{a_{i,j} \, | \, 1 \le j \le \Delta\}$, and $A = \{a_{i,j} \, | \, 1 \le i \le t, \; 1 \le j \le \Delta\}$ it is clear that $G$ is of the form described in the statement of the theorem. For the converse, consider a bipartite graph $G$ with partite sets $A$ and $B$, $|A| \le |B|$, where $A = \{a_{i,j} \, | \, 1 \le i \le t, \; 1 \le j \le \Delta\}$ and $S = \{s_i \,| \, 1 \le i \le t\} \subseteq B$, such that the following holds: $N(s_i) = \{a_{i,j} \, | \, 1 \le j \le \Delta\}$, $|N(a) \cap B| = \Delta$ for all $a \in A$ and $|N(b) \cap A| \le \Delta$ for all $b \in B$. Clearly, $S$ is an isolating set of $G$. Moreover, $m(G) = \Delta |A| = \Delta^2 t$. Hence, by the above inequality, $\Delta^2 t = m(G) \le \Delta^2 \iota(G) \le \Delta^2 |S| = \Delta^2 t$. Thus, $\iota(G) = \frac{m(G)}{\Delta^2} = \frac{dn}{2\Delta^2}$.
\end{proof}

Observe that, for an $r$-regular graph $G$ of order $n$ and attaining the bound of Theorem \ref{thm-lb-degrees}, we have $\gamma(G) = \frac{n}{r+1}$, while $\iota(G) = \frac{n}{2r}$. Hence, we have here another example where the parameters $\iota(G)$ and $\gamma(G)$ differ considerably, namely here on a factor of $\frac{1}{2}$.


\section{Bounds for certain families of graphs}

In the following two theorems, we will deal with the $k$-isolation number of trees. Given a tree $T$, we will call a vertex $x \in V(T)$ a \emph{support vertex} of $T$ if $x$ is neighbor of a leaf, i.e. a vertex of degree one. Moreover, an \emph{inner vertex} of $T$ is a vertex that is not a leaf.

\begin{thm}
Let $T$ be a tree on $n$ vertices and different from $K_{1,k+1}$. Then $\iota_k(T) \le \frac{n}{k+3}$ and this is sharp.
\end{thm}

\begin{proof}
Let $T$ be a tree different from $K_{1,k+1}$. We will prove the statement by induction on the number of vertices of $T$. If $n \le k+2$, then clearly $T$ has no $K_{1,k+1}$ as a subgraph and $\iota_k(G) = 0 \le \frac{n}{k+3}$. Suppose now that $T$ is a tree on $n \ge k+3$ vertices and assume that, for any tree on less than $n$ vertices and different from $K_{1,k+1}$, the above inequality holds. We now distinguish two cases.\\
{\it Case 1. Suppose that $T$ has a support vertex $v$ of degree $\deg(v) \neq k+1$.} \\
Let $u$ be a leaf adjacent to $v$ and define $T' = T - u$. Then $T'$ is a tree on $n-1$ vertices. If $T'= K_{1,k+1}$, then $T = K_{1,k+2}$ or $T$ is isomorphic to a $K_{1,k+1}$ with a subdivided edge. Since in both cases $\{v\}$ is a $k$-isolating set of $T$ and $n = k+3$, the inequality $\iota_k(T) \le \frac{n(T)}{k+3}$ holds trivially. Hence, we may assume that $T' \neq K_{1,k+1}$ and, by the induction hypothesis, $\iota_k(T') \le \frac{n(T')}{k+3}$. Now, due to the degree condition on $v$, observe that $u$ cannot belong to any subtree isomorphic to $K_{1,k+1}$ and, thus, any $k$-isolating set of $T'$ is also a $k$-isolating set of $T$. Hence, $\iota_k(T) \le \iota_k(T') \le \frac{n(T')}{k+3} < \frac{n}{k+3}$ and we are done.\\
{\it Case 2. Suppose that all support vertices of $T$ have degree $k+1$. }\\
Observe first that the diameter of $T$ cannot be less than $3$: otherwise, $T$ would be a star $K_{1,r}$ and since the support vertices have all degree $k+1$, $r = k+1$ and thus $T = K_{1,k+1}$, which is a contradiction to the assumptions. Hence, ${\rm diam}(T) \ge 3$. Let $P = x_0x_1 \ldots x_d$ be a diametral path of $T$. If $d = {\rm diam}(T) = 3$, it is easy to see that $\{x_1\}$ is a $k$-isolating set of $T$ and so $\iota_k(G) \le 1 \le \frac{n(T)}{k+3}$. Hence, we may assume that $d = {\rm diam}(T) \ge 4$. Let $T_1$ and $T_2$ be the trees resulting after removing the edge $x_3x_4$ from $T$, where $T_1$ is the tree containing $x_3$ and $T_2$ is the tree containing $x_4$. Since all support vertices of $T$ have degree $k+1$, $T_1$ has at least $k+3$ vertices and $\{x_3\}$ is a $k$-isolating set of $T_1$. If, further, ${\rm diam}(T_2) \le 2$, then $\{x_3\}$ is a $k$-isolating set of $T$ itself and clearly $\iota_k(T) \le 1 \le \frac{n}{k+3}$. On the other side, if ${\rm diam}(T_2) \ge 3$, then $T_2 \neq K_{1,k+1}$ and, by the induction hypothesis, we have $\iota_k(T_2) \le \frac{n(T_2)}{k+3}$. Now let $I$ be a minimum $k$-isolating set of $T_2$. Then $I \cup \{x_3\}$ is a $k$-isolating set of $T$, which implies
\[\iota_k(T) \le |I| + 1 = \iota_k(T_2) + 1 \le \frac{n(T_2)}{k+3} + 1 \le \frac{n(T_2)}{k+3} + \frac{n(T_1)}{k+3} = \frac{n}{k+3},\]
and we are done.\\
For the sharpness, consider a path $P$ on $t$ vertices $v_1,v_2, \ldots, v_{t}$ and $t$ copies of $K_{1,k+1}$ such that each vertex $v_i$ of $P$ is joined by an edge to one of the leaves of the $i$-th copy of $K_{1,k+1}$. Define the in this way constructed tree by $T$. Clearly, the vertices of $P$ are a $k$-isolating set of $T$ and we cannot come out with less since for each copy of $K_{1,k+1}$ we need at least one vertex in the $k$-isolating set.
\end{proof}

\begin{thm}
Let $T$ be a tree on $n$ vertices in which all non-leaves have equal degree $r \ge k+3$. Then $\iota_k(T) \le \frac{n-2}{2(r-1)}$ and this is sharp.
\end{thm}

\begin{proof}
Let $I$ be the set of inner vertices and $L$ the set of leaves of $T$. Since all inner vertices of $T$ have degree $r \ge k+3$, we have the following equality chain:
\[
2(n-1) = 2 m(T) = r|I| + |L| = r |I| + n - |I| = n + (r-1)|I|.
\]
This implies that $|I| = \frac{n-2}{r-1}$. Let now $D$ be minimum dominating set of the tree $T-L$, resulting from removing all leaves of $T$. Then $D$ is a $k$-isolating set of $T$ and we have, with Lemma \ref{la-dom-sum}~(iii) and Ore's inequality,
\[
\iota_k(T) \le \gamma(T-L) \le \frac{|I|}{2} = \frac{n-2}{2(r-1)}.
\]
For the sharpness, construct a tree $T$ the following way. Take a path $P = x_1x_2 \ldots x_t$ on $t \ge 3$ vertices and attach to each of the inner vertices $x_i$ of $P$ a leaf $v_i$, where $2 \le i \le t-1$. Attach now, to each vertex $v_i$, $r-1$ leaves and, to each vertex $x_i$, $r-2$ leaves, for $2 \le i \le t-1$. Then, any $k$-isolating set of the resulting tree $T$ contains either $x_i$ or $v_i$ for each $2 \le i \le t-1$. Moreover, $\{x_2, x_3, \ldots, x_{t-1}\}$ is a $k$-isolating set of $T$. Hence, $\iota_k(T) = t-1 = \frac{n(T) -2}{2(r-1)}$.
\end{proof}
 
For the theorem coming next, we shall need the following result from  Campos and Wakabayashi \cite{CamWak}. 

\begin{thm}[\cite{CamWak}]\label{thm_camwak}
Let $G$ be a maximal outerplanar graph on $n \ge 4$ vertices and $t$ vertices of degree $2$. Then, $\gamma(G) \le \frac{n + t}{4}$.
\end{thm}

Next theorem shows that, for a maximal outerplanar graph, at most $\frac{1}{4}$ of the vertices are needed for an isolating set.

\begin{thm}
Let $G$ be a maximal outerplanar graph on $n \ge 4$ vertices. Then $\iota(G) \le \frac{n}{4}$ and this is sharp.
\end{thm}

\begin{proof}
If $n = 4$, then $G$ is the complete graph on $4$ vertices minus an edge, which has one vertex dominating all others and thus $\iota(G) = 1$. If $n = 5$, then again, since $G$ is a triangulation of the pentagon, there has to be a vertex dominating all others. When $n = 6, 7$, we use the fact that a maximal outerplanar graph has $m = 2n - 3$ edges. It is also well known that a maximal outerplanar graph has at least two vertices of degree $2$. If $n = 6$, these two facts imply that there has to be a vertex of degree at least $4$, otherwise we would have $18 = 2m = \sum_{v \in V(G)} \deg(v) \le 2 \cdot 2 + 4 \cdot 3 = 16$, a contradiction. Since a vertex of degree $4$ in a $6$-vertex graph forms an isolating set, we have $\iota(G) = 1$. Let now $n = 7$. Note that the only possible degree sequences of $7$ vertices satisfying $\sum_{v \in V(G)} \deg(v) =2m  = 2(2n-3) = 22$ and having at least $2$ vertices of degree $2$ are $5,4,4,3,2,2,2$ and $4,4,4,3,3,2,2$. Let $C= v_1v_2v_3v_4v_5v_6v_7v_1$ be the cycle  surrounding the outerface of $G$ and suppose that $v_1$ has degree $4$. If $v_1$ is not-adjacent to two non-consecutive vertices of the cycle different from $v_2$ and $v_7$, then $\{v\}$ is an isolating set. Hence suppose that $v_1$ is not-adjacent to two consecutive vertices on the cycle different from $v_2$ and $v_7$. Without loss of generality, due to symmetry reasons, we can assume that either $N(v) = \{v_2,v_3,v_4,v_7\}$ or $N(v) = \{v_2,v_3,v_6,v_7\}$. In the first case, since $G$ is a triangulation of $C$, $v_4$ has to be adjacent to $v_7$. Then $v_4$ is a vertex of degree at least $4$ which dominates all but at most the two non-adjacent vertices $v_2$ and $v_6$, and hence $\iota(G) = 1$. In the second case, i.e. when $v_1$ is adjacent to $v_3$ and $v_6$, then $v_3v_6 \in E(G)$. Then either $v_3v_5$ or $v_4v_6 \in E(G)$, implying that either $v_3$ or $v_6$ has degree $5$, which leads to $\iota(G) = 1$.

Let now $G$ be a maximal outerplanar graph on $n \ge 8$ vertices. Let $N_2$ be the set of vertices of degree $2$ in $G$ and let $n_2 = |N_2|$. Note that (for $n \ge 4$) $N_2$ is an independent set. Hence, $n_2 \le \frac{n}{2}$. Let $G^* = G - N_2$. Note  that (for $n \ge 5$) the deletion of any vertex from $N_2$ creates at most one new vertex of degree $2$. Hence, $G^*$ has at most $n_2^* \le n_2$ vertices of degree $2$. Further, $G^*$ is a maximal outerplanar graph on $n^* \ge \frac{n}{2} \ge 4$ vertices. By Theorem \ref{thm_camwak}, it follows that $\gamma(G^*) \le \frac{n^* + n_2^*}{4}$. Hence, by Lemma \ref{la-dom-sum} (iii), $\iota(G) \le \gamma(G \setminus G[N_2]) = \gamma(G^*) \le \frac{n^* + n_2^*}{4} \le \frac{n^* + n_2}{4} = \frac{n}{4}$ and we are done.

To see the sharpness, consider an arbitrary maximal outerplanar graph on $2p$ vertices such that $v_1v_2\ldots v_{2p}v_1$ is the outercycle. Insert new vertices $w_i$ and edges $v_iw_i$, for $1 \le i \le 2p$, and also the edges $v_{2j-1}w_{2j}$ and $w_{2j-1}w_{2j}$, for $1 \le j \le p$. Then we have constructed again a maximal outerplanar graph $G$ on $n = 4p$ vertices, now with the outercycle going along the paths $v_{2j-1}w_{2j-1}w_{2j}v_2j$ in consecutive order, for $1 \le j \le p$. Note that, any isolating set of $G$ has to contain at least one vertex from each of these paths, otherwise the adjacent vertices $w_{2j-1}$ and $w_{2j}$ would not be dominated. Hence, $\iota(G) \ge \frac{n}{4}$. On the other hand, $\{w_{2j-1} \;|\; 1 \le j \le p\}$ is an isolating set of $G$ with $p = \frac{n}{4}$ vertices. Thus, we have $\iota(G) = \frac{n}{4}$.
\end{proof}

In the following theorem, we consider \emph{claw-free graphs}, i. e. graphs which do not contain a $K_{1,3}$ as an induced subgraph.

\begin{thm}\label{thm-lb-K1,3-free}
Let $G$ be a claw-free graph on $n$ vertices with average degree ${\rm d}$, maximum degree $\Delta$ and minimum degree $\delta$. Then the following lower bounds on $\iota(G)$ hold:
\begin{enumerate} 
\item[(i)] $\iota(G) \ge \frac{\delta (n+1)+2}{(\delta +2)(\Delta+1)}$, 
\item[(ii)] $\iota(G) \ge \frac{2\, {\rm d }\, n}{3 \Delta^2+2 \Delta}$. 
\end{enumerate}
\end{thm}

\begin{proof}
(i) In \cite{FGJLL}, the authors prove the bound $\alpha(G) \ge \frac{r-1}{r-1+\delta} n$ for graphs on $n$ vertices and minimum degree $\delta$ and without induced $K_{1,r}$. Setting $r=3$ in this result, and combining it with the bound on Theorem \ref{thm-lb-power-alpha}~(ii), leads to
\[
\iota(G) \ge \frac{n+1 - \alpha(G)}{\Delta+1} \ge \frac{n+1- \frac{2n}{\delta+2}}{\Delta+1} = \frac{\delta(n+1)+2}{(\delta+2)(\Delta+1)},
\]
and we are done.\\
(ii) Let $G$ be claw-free and let $S$ be a minimum isolating set of $G$. Similarly as in Theorem \ref{thm-lb-degrees}, we will bound the number of edges of $G$ from above. Consider a vertex $x \in S$ and let $N_x = N(x) \setminus S$, $G_x = G[N[N_x] \cup \{x\}]$ and $t_x = |N_x|$. Differently from the proof of Theorem \ref{thm-lb-degrees}, since $G$ is claw-free, $N(v)$ has to contain enough edges such that any independent set of three vertices is avoided. This holds in particular also for $N_x$. The minimum number of edges contained in a graph with independence number at most $2$ is equal to the number of edges of the complement of the Tur\'an graph (the triangle-free ones). That is, in the worst case, $G[N_x]$ consists of two cliques of equal or almost equal order, depending on the parity of $t_x$. Hence, $G[N_x]$ has at most ${\lceil t_x/2 \rceil \choose 2} + {\lfloor t_x/2 \rfloor \choose 2}$ edges. 
Hence, we have the following.
\begin{align*}
m(G_x) &= t_x + m(G_x - x)\\
& = t_x + \sum_{v \in N_x} \deg_{G_x-x}(v) - m(G[N_x])\\
& \le t_x + t_x (\Delta -1) - m(G[N_x])\\
&= t_x \Delta - m(G[N_x]).
\end{align*}
If $t_x$ is odd, we have $m(G[N_x]) \ge {(t_x+1)/2 \choose 2} + { (t_x-1)/2 \choose 2} = \frac{(t_x-1)^2}{4}$ and thus
\[
m(G_x) \le t_x \Delta - \frac{(t_x-1)^2}{4} \le \Delta^2 - \frac{(\Delta-1)^2}{4} = \frac{3 \Delta^2+2 \Delta-1}{4} < \frac{3 \Delta^2+2 \Delta}{4}.
\]
If $t_x$ is even, we have $m(G[N_x]) \ge  2 {t_x/2 \choose 2} = \frac{t_x(t_x-2)}{4}$ and thus
\[
m(G_x) \le t_x \Delta - \frac{t_x(t_x-2)}{4} \le \Delta^2 - \frac{\Delta (\Delta-2)}{4}= \frac{3 \Delta^2+2 \Delta}{4}.
\]
We obtain now the following upper bound on the number of edges of $G$:
\[m(G) \le \sum_{x \in S} m(G_x) \le |S| \frac{3 \Delta^2+2 \Delta}{4},
\]
which implies $\iota(G) = |S| \ge \frac{4 m(G)}{3 \Delta^2+2 \Delta} = \frac{2 {\rm d}}{3 \Delta^2+2 \Delta} n$.
\end{proof}

Observe that, for an $r$-regular claw-free graph $G$ on $n$ vertices, the bound of Theorem \ref{thm-lb-K1,3-free} gives $\iota(G) \ge \frac{2n}{3r+2}$, which is on around a factor of $\frac{4}{3}$ better than the bound $\iota(G) \ge \frac{n}{2r}$ of Theorem \ref{thm-lb-degrees} for $r$-regular graphs.

In the following theorem, we shall consider grids and other similar graphs. 

\begin{thm}\label{thm-grids}
Let $t, s \ge 3$ be two integers. Then
\begin{align*}
\frac{st}{8} &\le \iota(C_s \times C_t) \le \frac{st}{8} + \frac{3(s+t+3)}{8},\\
\frac{st}{8} - \frac{t}{16} &\le\iota(P_s \times C_t) \le  \frac{st}{8} + \frac{3s+t+3}{8},\\
\frac{st}{8} - \frac{s+t}{16} &\le \iota(P_s \times P_t) \le  \frac{st}{8} + \frac{s+t+1}{8}.
\end{align*}
Further, the bound $\frac{st}{8} \le \iota(C_s \times C_t)$ is sharp for $s, t \equiv 0 \;({\rm mod}\; 4)$.
\end{thm}

\begin{proof}
Let $G = C_s \times C_t$. The lower bound follows directly from Theorem \ref{thm-lb-degrees} using the fact that $G$ is $4$-regular. For the sharpness, let $s,t \equiv 0 \;({\rm mod} \; 4)$ and consider the isolating set consisting of all vertices on a row and column congruent to $1 \;({\rm mod}\; 4)$ or in a row and column congruent to $3\; ({\rm mod}\; 4)$. For the upper bound, we will construct an isolating set the following way. Let $V(G) = \{ (i,j) \;|\; 1 \le i \le s, 1 \le j \le t\}$ be the set of vertices of $G$ where, for $A_i = \{(i,j) \;|\; 1 \le j \le t\}$ and $B_j = \{ (i,j) \;|\;  1 \le i \le s\}$, $G[A_i] \cong C_t$ and $G[B_j] \cong C_s$. Let $S_1 = \{(i,j) \;|\; i,j \equiv 1 \;({\rm mod}\; 4)\}$ and $S_2 = \{(i,j) \;|\; i,j \equiv 3 \;({\rm mod} \; 4)\}\}$. Then $S = S_1 \cup S_2$ is an isolating set of $G$ with 
\[\iota(G) \le |S| \le 2 \left\lceil \frac{s}{4} \right\rceil \left\lceil \frac{t}{4} \right\rceil \le \frac{(s+3)(t+3)}{8} = \frac{st}{8} + \frac{3(s+t+3)}{8}.\]
Now let $H = P_s \times C_t$. Then $H$ has $2t$ vertices of degree $3$ and $t(s-2)$ vertices of degree $4$. Hence, ${\rm d}(G) = \frac{6t + 4t(s-2)}{st} = \frac{2t(2s-1)}{st}$ and Theorem \ref{thm-lb-degrees} yields $\iota(H) \ge \frac{t(2s-1)}{16} = \frac{st}{8} - \frac{t}{16}$. For the upper bound, define the vertices as in the previous case such that $G[A_i] \cong C_t$ and $G[B_j] \cong P_s$ and set $S_1 = \{(i,j) \;|\; i,j \equiv 1 \;({\rm mod}\; 4)\}$ and $S_2 = \{(i,j) \;|\; i,j \equiv 3 \;({\rm mod} \; 4)\}\}$. Further, define 
\[
S^* = \left \{ \begin{array}{ll}
                \{(s,j) \;|\; j \equiv 1 \;({\rm mod}\; 4)\}, & \mbox{if } s \equiv 2 \;({\rm mod}\; 4)\\
                \{(s,j) \;|\; j \equiv 3 \;({\rm mod}\; 4)\}, & \mbox{if } s \equiv 0 \;({\rm mod}\; 4)\\
                \emptyset,  & \mbox{if } s \equiv 1,3 \;({\rm mod}\; 4)
\end{array} \right.
\]
and $S = S_1 \cup S_2 \cup S^*$.
Then $S$ is an isolating set of $H$. Observe that $S$ has $\lceil \frac{s}{2}\rceil$ rows of $\lceil \frac{n}{4}\rceil$ vertices each, and thus we obtain
\[\iota(G) \le |S| \le \left\lceil \frac{s}{2}\right\rceil \left\lceil \frac{t}{4} \right\rceil \le \frac{(t+3)(s+1)}{8} = \frac{st}{8} + \frac{3s+t+3}{8}.\]
Finally, let $J = P_s \times P_t$. Observe that $J$ has $2$ vertices of degree $2$, $2s+2t-8$ vertices of degree $3$ and $(s-2)(t-2)$ vertices of degree $4$. This gives ${\rm d}(G) = \frac{4st-2s-2t}{st}$ and thus we have with Theorem \ref{thm-lb-degrees} $\iota(G) \ge \frac{2st-s-t}{16} = \frac{st}{8} - \frac{s+t}{16}$. For the upper bound, define the vertices as above such that $G[A_i] \cong P_t$ and $G[B_j] \cong P_s$ and set $S_1 = \{(i,j) \;|\; i,j \equiv 1 \;({\rm mod}\; 4)\} \cup \{(s,j) \;|\; j \equiv 1 \;({\rm mod}\; 4)\}$ and $S_2 = \{(i,j) \;|\; i,j \equiv 3 \;({\rm mod} \; 4)\}\} \cup \{(i,t) \;|\; i \equiv 3 \;({\rm mod}\; 4) \}$. Further, define 
\[
S^* = \left \{ \begin{array}{ll}
                \{(s,j) \;|\; j \equiv 1 \;({\rm mod}\; 4)\}, & \mbox{if } s \equiv 0 \;({\rm mod}\; 4)\\
                \{(s,j) \;|\; j \equiv 3 \;({\rm mod}\; 4)\}, & \mbox{if } s \equiv 2 \;({\rm mod}\; 4)\\
                \emptyset,  & \mbox{if } s \equiv 1,3 \;({\rm mod}\; 4),
\end{array} \right.
\]
\[
T^* = \left \{ \begin{array}{ll}
                \{(i,t) \;|\; i \equiv 1 \;({\rm mod}\; 4)\}, & \mbox{if } t \equiv 0 \;({\rm mod}\; 4)\\
                \{(i,t) \;|\; j \equiv 3 \;({\rm mod}\; 4)\}, & \mbox{if } t \equiv 2 \;({\rm mod}\; 4)\\
                \emptyset,  & \mbox{if } t \equiv 1,3 \;({\rm mod}\; 4),
\end{array} \right.
\]
and $S = S_1 \cup S_2 \cup S^* \cup T^*$. Then $S$ has $\lceil \frac{s}{4}\rceil$ rows of vertices $(i,j)$ with $i, j \equiv 1 \;({\rm mod}\; 4)$, each having $\lceil \frac{t}{4}\rceil$ vertices. Similarly, $S$ has $\lceil \frac{s-2}{4}\rceil$ rows of vertices $(i,j)$ with $i, j \equiv 3 \;({\rm mod}\; 4)$, each having $\lceil \frac{t-2}{4}\rceil$ vertices. Hence,
\[\iota(G) \le |S| \le \left\lceil \frac{s}{4} \right\rceil \left\lceil \frac{t}{4} \right\rceil + \left\lceil \frac{s-2}{4} \right\rceil \left\lceil \frac{t-2}{4} \right\rceil.\]
The last inequality is worst when $s, t \equiv 3 \;({\rm mod}\; 4)$, and thus
\[\iota(G) \le 2\frac{(s+1)(t+1)}{16} = \frac{st}{8} + \frac{s+t+1}{8}.\]
\end{proof}

Observe that, according to \cite{ACIOP, GPRT}, for $s \ge t \ge 16$, the domination number of the grid graph $P_s \times P_t$ is equal to $\frac{st}{5} + \mathcal{O}(s+t)$, which is around a factor of $\frac{8}{5}$ larger than the bounds for the isolation number of grid graphs from Theorem \ref{thm-grids}.

\section{Nordhaus-Gaddum bounds}

In this section, we will deal with Nordhaus-Gaddum bounds for the isolation number. Observe that $\iota(G) = 0$ if and only if $G = \ov{K_n}$. Since, for $n \ge 2$, $\iota(K_n) = 1$ it follows that $\iota(G)+ \iota(\ov{G}) \ge 1$ for $n \ge 2$ and this is sharp. The following results give upper bounds on the sum $\iota(G) + \iota(\ov{G})$.

\begin{thm}\label{thm:iota>3}
Let $G$ be a graph with minimum degree $\delta$. Then, if $\iota(\ov{G}) \ge 3$, $\iota(G)+ \iota(\ov{G}) \le \delta+1$.
\end{thm}

\begin{proof}
Observe first that if $\delta(G) \le 3$, then $i(\ov{G}) \le 2$, since if there is a vertex $y \in V$ with $\deg(y)\le 3$, then, in $\ov{G}$, $|V \setminus N_{\ov{G}}[y]| \le 3$ and $V \setminus N_{\ov{G}}[y]$ can be isolated by one more vertex. So we may assume, without loss of generality, that $\delta(G) \ge 4$. The condition $\iota(\ov{G}) \ge 3$ implies that no two vertices form an isolating set in $\ov{G}$, meaning that, for any two vertices $x, y \in V$, there is an edge $uv \in E(\ov{G})$ with both end vertices $u, v$ outside $N_{\ov{G}}[x] \cup N_{\ov{G}}[y]$. In particular, it follows that, in $G$, $|N_G(x) \cap N_G(y)| \ge 2$ for any two vertices $x, y \in V$ and thus ${\rm diam(G)} \le 2$. Since $\iota(\ov{G}) \ge 3$ we have also that $G \neq K_n$ and, hence, we obtain that ${\rm diam}(G) = 2$. Let now $u$ be a vertex of minimum degree $\delta$ in $G$. Then $N_G(u)$ is a dominating and an isolating set in $G$. Let $Y = V \setminus N_G[u]$ and $X' = N_G(u) \setminus X$. Note that $V = X \cup X' \cup Y \cup \{u\}$ is a disjoint union. Let $x \in X$. Then, by the minimality condition on $X$, there are vertices $y, z \in Y \cup X' \cup \{x\}$ such that $yz \in E(G)$ and $(N_G(y) \cup N_G(z)) \cap X = \{x\}$. Hence, in $\ov{G}$, there are vertices $y, z \in Y \cup X' \cup\{x\}$ such that $yz \notin E(\ov{G})$ and $X \setminus \{x\} = (N_{\ov{G}}(y) \cap N_{\ov{G}}(z)) \cap X$. Observe that $X' \neq \emptyset$, since otherwise $\{u,y\}$ would be an isolating set in $\ov{G}$, contradicting $\iota(\ov{G}) \ge 3$. Let $I \subseteq X'$ be the set of all isolated vertices in $\ov{G}[X']$ and let $D$ be a minimum isolating set in $\ov{G}[X'\setminus I]$ (if $X' \setminus I = \emptyset$, set $D = \emptyset$). Then $|D| \le \frac{|X' \setminus I|}{2}$. We distinguish now the following cases.\\
{\it Case 1: $I = \emptyset$.} Then $|X'| \ge 2$ and thus $|X| \le \delta-2$. In this case, $D \cup \{u,y\}$ is an isolating set of $\ov{G}$ and therefore we have
\begin{align*}
\iota(G) + \iota(\ov{G}) &\le |X| + |D \cup \{u, y\}| \\
                         &\le |X| + \frac{|X'|}{2} +2 \\
                         & = |X| + \frac{\delta - |X|}{2} +2\\
                         &= \frac{|X|+\delta}{2}+2 \\
                         &\le \frac{2\delta-2}{2}+2 = \delta+1.
\end{align*}
{\it Case 2: $|I| = 1$, say $I = \{w\}$.} Then $|X| = \delta - |X'| \le \delta -1$. If $N_G(w) \cap X = \emptyset$, then $X \subseteq N_{\ov{G}}(w)$ and, thus, $D \cup \{u,w\}$ is an isolating set in $\ov{G}$. Hence,
\begin{align*}
\iota(G) + \iota(\ov{G}) &\le |X| + |D \cup \{u, w\}|\\
                         &\le |X| + \frac{|X' \setminus I|}{2} +2 \\
                         &  = |X| + \frac{\delta - |X| - |I|}{2} +2 \\
                         & = \frac{\delta + |X| - |I|}{2}+2 \\
                         &\le \frac{2 \delta - 2}{2}+ 2 = \delta+1
\end{align*}
and we are done. Therefore, we may assume, without loss of generality, that $wx \in E(G)$. In this case, $D \cup \{u,y\}$ is an isolating set in $\ov{G}$. As above, we obtain again $\iota(G) + \iota(\ov{G}) \le \delta+1$.\\
{\it Case 3: $|I| \ge 2$.} Then $|X| \le \delta-2$ and $D \cup \{u,x,y\}$ is an isolating set of $\ov{G}$. Thus we have
\begin{align*}
\iota(G) + \iota(\ov{G}) &\le |X| + |D \cup \{u,x, y\}| \\
                         &\le |X| + \frac{|X' \setminus I|}{2} +3 \\
                         &\le |X| + \frac{\delta - |X| - |I|}{2} + 3\\
                         &=  \frac{\delta + |X| - |I|}{2}+3 \\
                         &\le \frac{2 \delta - 4}{2}+3 = \delta+1
\end{align*}
and we are done.
\end{proof}

\begin{thm}\label{thm:NG-f(delta)}
Let $f(x)$ be a function defined for $x \in [2, \infty)$ such that $\iota(G) \le f(\delta) n$ for all graphs $G$ of order $n$ and minimum degree $\delta \ge 2$. Then, for any graph $G$ of order $n \ge \frac{\delta-1}{f(\delta)}$, we have 
\[\iota(G) + \iota(\ov{G}) \le n f(\delta) + 2.\]
Moreover, if the upper bound $f(\delta) n$ is attained for a graph $G$ of order $n$ and minimum degree $\delta \ge 2$, then the above inequality is sharp.
\end{thm}

\begin{proof}
Let $G$ be a graph of order $n$ and minimum degree $\delta \ge 2$. By Theorem \ref{thm:iota>3}, if $\iota(\ov{G}) \ge 3$, we have $\iota(G)+ \iota(\ov{G}) \le \delta+1$. On the other side, if $\iota(\ov{G}) \le 2$, we obtain $\iota(G)+ \iota(\ov{G}) \le f(\delta) n+2$. Hence, $\iota(G)+ \iota(\ov{G}) \le \max \{\delta+1, f(\delta) n + 2 \}$. If $n \ge \frac{\delta-1}{f(\delta)}$, then $f(\delta)n + 2 \ge \delta + 1$, implying thus $\iota(G) + \iota(\ov{G}) \le n f(\delta) + 2$. For the sharpness, assume that $\iota(G) = f(\delta) n(G)$ for a graph $G$ of minimum degree $\delta \ge 2$. Let $H$ be the graph consisting of $\frac{n}{n(G)} \ge 2$ copies of $G$, where $n$ is an integer divisible by $n(G)$. Then $\iota(H) = \frac{n}{n(G)} f(\delta)n(G) = f(\delta)n$ and $\iota(\ov{G}) = 2$.
\end{proof}

\begin{rem}
Observe that, if $\delta(G) \le 1$, then $\iota(\ov{G}) = 1$.  Hence, in this case, we have $\iota(G) + \iota(\ov{G}) = \iota(G) + 1$, showing that the bound of Theorem \ref{thm:NG-f(delta)} cannot be attained.
\end{rem}

\begin{cor}
Let $G$ be a graph of oder $n$ and minimum degree $\delta$. Then we have the following bounds.
\begin{enumerate}
\item[(i)] If $\delta = 0$, $\iota(G) + \iota(\ov{G}) \le \left\lfloor \frac{n+1}{2} \right\rfloor$ and this is sharp.
\item[(ii)] If $\delta = 1$, $\iota(G) + \iota(\ov{G}) \le \left\lfloor\frac{n}{2} \right\rfloor + 1$ and this is sharp.
\item[(iii)] If $\delta = 2$, $\iota(G) + \iota(\ov{G}) \le \frac{2}{5}n + 2$ and this is sharp.
\item[(iv)] If $\delta = 3$, $\iota(G) + \iota(\ov{G}) \le \frac{1}{3}n + 2$ and this is sharp.
\item[(v)] If $n \ge \frac{(\delta-1)(\delta+1)}{\ln(\delta+1)+ \frac{1}{2}}$, then $\iota(G) + \iota(\ov{G}) \le \frac{\ln(\delta+1)+ \frac{1}{2}}{\delta+1} n +2.$
\end{enumerate}
\end{cor}

\begin{proof}
(i) Let $G$ be a graph on $n$ vertices and with minimum degree $\delta = 0$. Let $I$ be the set of isolated vertices in $G$. Then, by Theorem \ref{thm:up-bound-min-deg}(ii), $\iota(G) = \iota(G - I) \le \left\lfloor\frac{n - |I|}{2} \right\rfloor\le \left\lfloor\frac{n-1}{2}\right\rfloor$. Since in $\ov{G}$ there is a vertex of degree $n - 1$, $\iota(\ov{G}) = 1$ and, hence, $\iota(G) + \iota(\ov{G}) \le \left\lfloor\frac{n-1}{2} \right\rfloor+1 = \left\lfloor\frac{n+1}{2}\right\rfloor$. The sharpness follows by considering the graph $G = \frac{n-1}{2}K_2 \cup K_1$, when $n$ is odd, and $G = \frac{n-2}{2}K_2 \cup 2K_1$, when $n$ is even.\\
(ii) Let $G$ be a graph on $n$ vertices and minimum degree $\delta = 1$. Then, by Theorem \ref{thm:up-bound-min-deg} (ii), we have $\iota(G) \le \left\lfloor\frac{n}{2}\right\rfloor$ and, since $\ov{G}$ has a vertex of degree $n-2$, $\iota(\ov{G}) = 1$. Hence, $\iota(G) + \iota(\ov{G}) \le \left\lfloor\frac{n}{2}\right\rfloor+1$. For the sharpness, consider the graph $G = \frac{n}{2}K_2$, when $n$ is even, and $G = \frac{n-3}{2}K_2 \cup K_{1,2}$, when $n$ is odd.\\
(iii) - (v) The bounds follow from Theorems \ref{thm:up-bound-min-deg} and \ref{thm:NG-f(delta)} for $n \ge \frac{\delta-1}{f(\delta)}$, with $f(2) = \frac{2}{5}$, $f(3) = \frac{1}{3}$ and $f(\delta) = \frac{\ln(\delta+1)+ \frac{1}{2}}{\delta+1}$. For $\delta \in \{2,3\}$ and small values of $n$, that is, $\delta+ 1 \le n < \frac{\delta-1}{f(\delta)}$, we only need to check when $\delta = 3$ and $n = 4$ or $5$. In this case, the only possibilities for $G$ are either $K_4$, $K_5 - e$, or $K_5 - \{e,f\}$, where $e$ and $f$ are the edges of a matching in $K_5$. Evidently, the bound holds in these cases, too.
\end{proof}


\section{Open problems}\label{problems}


This paper, as it is introducing a new subject, offers obviously many possible directions, open problems, conjectures and generalizations. Rather than doing this, we choose to offer some few concrete open problems where  a progress to solve them seems doable, and which hopefully will shade more light on which ideas and techniques will be useful in attacking such problems on partial domination with restricted structure imposed on the non-dominated vertices.\\

Recalling Theorem \ref{thm:up-bound-min-deg}, we show $f(\delta)  \ge \frac{2}{\delta+3}$ and we prove that this is sharp for $\delta = 1,2,3$.

\begin{prob} 
Determine other values of $f(\delta)$. In particular, is $f(4) = \frac{2}{7}$?
\end{prob}
 
Recall that, by means of Theorem \ref{thm-gral}, we show in Corollary \ref{cor-f_c} that $f_c(1) = f_c(2) = \frac{1}{3}$ and that $f_c(\delta) \ge \frac{2}{\delta+4}$ for $\delta \ge 3$.  
 
\begin{prob} 
Determine other values of $f_c(\delta)$. In particular, is $f_c(3) = \frac{2}{7}$?
\end{prob}

For $\mathcal{F} = \{C_k \;|\;  k\ge 3 \}$, we know after Theorem \ref{thm-n/3} and Corollary \ref{cor-conn-isol} that $\frac{1}{4} \le \limsup_{n \rightarrow \infty} \{ \frac{\iota(G, \mathcal{F})}{n(G)} \;|\; n(G) \ge n \}  \le \frac{1}{3}$. Also for $\mathcal{F} = \{K_{k+1}\}$ and
$\mathcal{F} = \mathcal{F}_k$, the family of all trees of order $k \ge 2$, we know, after the same theorem and corollary, that $\frac{1}{k+1} \le \limsup_{n \rightarrow \infty} \{ \frac{\iota(G, \mathcal{F})}{n(G)} \;|\; n(G) \ge n \}  \le \frac{1}{3}$. Hence, we state the following problem.

\begin{prob} 
For $\mathcal{F} = \{C_k \;|\;  k\ge 3 \}$, $\mathcal{F} = \{K_{k+1}\}$
and $\mathcal{F} = \mathcal{F}_k$, determine  
\[\limsup_{n \rightarrow \infty} \left\{ \frac{\iota(G, \mathcal{F})}{n(G)} \;|\;  n(G) \ge n \right\} \mbox{ and } \sup_{n \rightarrow \infty} \left\{ \frac{\iota(G, \mathcal{F})}{n(G)} \;|\;  n(G) \ge n \right\}.\]
\end{prob}

\begin{prob}
Determine or give a lower and an upper bound for $\iota(G)$ or $\iota(G, \mathcal{F})$ for further interesting families of graphs.
\end{prob}
 
\begin{prob}
Estimate  $g(n,\delta) =  \max \{ \gamma(G) - \iota(G) \;|\;  n(G) = n, \; \delta(G) = \delta \}$.  Is it true that  $g(n,\delta)   \ge  cn \frac{ln (\delta+1)} {\delta+1}$ for some constant $c$ with $0 < c  < 1$? 
\end{prob}


\begin{thebibliography}{50}

\bibitem{ACIOP} S. Alanko, S. Crevals, A. Isopoussu, P. \"Osterg\r{a}rd, V. Pettersson, Computing the domination number of grid graphs. Electron. J. Combin. {\bf 18} (2011), no. 1, Paper 141, 18 pp.

\bibitem{AlSp} N. Alon, J. H. Spencer, \emph{The probabilistic method}. Third edition. With an appendix on the life and work of Paul Erd\H{o}s. Wiley-Interscience Series in Discrete Mathematics and Optimization. John Wiley \& Sons, Inc., Hoboken, NJ, 2008. xviii+352 pp.

\bibitem{AlWo} N. Alon, N. Wormald, High degree graphs contain large-star factors. Fete of combinatorics and computer science, 9–21, Bolyai Soc. Math. Stud. {\bf 20}, J\'anos Bolyai Math. Soc., Budapest (2010).

\bibitem{Arn} V. I. Arnautov, Estimations of the external stability number of a graph by means of the minimal degree of vertices, Prikl. Mat. Programm. {\bf V 11} (1974), 3--8 (in Russian).

\bibitem{CamWak} C. N. Campos, Y. Wakabayashi, On dominating sets of maximal outerplanar graphs, \emph{Discrete Appl. Math.} {\bf 161} (2013), no. 3, 330--335.

\bibitem{CaHa} Y. Caro, A. Hansberg, New approach to the k-independence number of a graph, \emph{Electron. J. Combin.} {\bf 20} (2013), no. 1, Paper 33, 17 pp.


\bibitem{CaHaHe} Y. Caro, A. Hansberg, M. A. Henning, Fair domination in graphs, \emph{Discrete Math.} {\bf 312} (2012), no. 19, 2905--2914. 


\bibitem{CWY} Y. Caro, D. B. West, R. Yuster, Connected domination and spanning trees with many leaves, \emph{SIAM J. Discrete Math.} {\bf 13} (2000), no. 2, 202--211 (electronic). 

\bibitem{ChFaHaVo} M. Chellali, O. Favaron, A. Hansberg,  L. Volkmann,
$k$-domination and $k$-independence in graphs: a survey,
\emph{Graphs Combin.} {\bf 28} 1 (2012), 1--55. 

\bibitem{CoTh} E. J. Cockayne, A. G. Thomason, An upper bound for the $k$-tuple domination number, \emph{J. Combin. Math. Combin. Comput.} {\bf 64} (2008), 251--254.

\bibitem{DHH} W. Desormeaux, T. W. Haynes, M. A. Henning, Bounds on the connected domination number of a graph. \emph{Discrete Appl. Math.} {\bf 161} (2013), no. 18, 2925--2931.

\bibitem{DeHe} W. J. Desormeaux, M. A. Henning, Paired domination in graphs: a survey and recent results, \emph{Util. Math.} {\bf 94} (2014), 101--166.

\bibitem{FGJLL} R. J. Faudree, R. J. Gould, M. S. Jacobson, L. M. Lesniak, T. E. Lindquester, On independent generalized degrees and independence numbers in $K(1,m)$-free graphs, \emph{Discrete Math.} {\bf 103} (1992), no. 1, 17--24.

\bibitem{FaHaVo} O. Favaron, A. Hansberg, L. Volkmann, On $k$-domination and minimum degree in graphs, \emph{J. Graph Theory} {\bf 57} (2008), no. 1, 33--40. 


\bibitem{Fou} F.~Foucaud. \emph{Combinatorial and algorithmic aspects of identifying codes in graphs}. PhD thesis, Universit\'{e} Bordeaux 1, France, December 2012. 


\bibitem{GaZv} A. Gagarin, V. E. Zverovich, A generalised upper bound for the $k$-tuple domination number. Discrete Math. 308 (2008), no. 5-6, 880–885. 



\bibitem{GPRT} D. Gon\c{c}alves, A. Pinlou, M. Rao, S. Thomass\'e, The domination number of grids, \emph{SIAM Journal of Discrete Mathematics} {\bf 25} (2011), 1443--1453.


\bibitem{Ha} A. Hansberg, \emph{Multiple domination in graphs}, PhD thesis, RWTH Aachen University (2009). 


\bibitem{He} M. A. Henning, Graphs with large total domination number, \emph{J. Graph Theory} {\bf 35} (1) (2000) 21--45. 

\bibitem{HeYe} M. A. Henning, A. Yeo, \emph{Total domination in graphs}, Springer Monographs in Mathematics, Springer, New York, 2013. xiv+178 pp.


\bibitem{KSKW} H. Karami, S. M. Sheikholeslami, A. Khodkar, D. B. West, Connected domination number of a graph and its complement, \emph{Graphs Combin.} {\bf 28} (2012), no. 1, 123--131.

\bibitem{LWAB} Y. Li, Y. Wu, Ch. Ai, R. Beyah, On the construction of $k$-connected $m$-dominating sets in wireless networks, \emph{J. Comb. Optim.} {\bf 23} (2012), no. 1, 118--139. 


\bibitem{Lov} L. Lov\'asz. On decompositions of graphs, Studia Sci. Math Hungar. {\bf 1} (1966) 237--238.

\bibitem{McCShe} W.~McCuaig and B.~Shepherd, Domination in graphs with minimum degree two, \emph{Journal of Graph Theory} {\bf 13}(6) (1989), 749--762.

\bibitem{Ore} O. Ore, Theory of graphs, \emph{American Mathematical Society
  Colloquium Publications}, Vol. XXXVIII, American Mathematical
  Society, Providence, R.I. (1962).
  
\bibitem{Pay}  C. Payan, Sur le nombre d'absorption d'un graphe simple, Cahiers Centre \'Etudes Recherche Op\'er. {\bf17} (1975), 307--317.


\bibitem{RaVo} D. Rautenbach, L. Volkmann, Lutz New bounds on the $k$-domination number and the $k$-tuple domination number, \emph{Appl. Math. Lett.} {\bf 20} (2007), no. 1, 98--102.

\bibitem{Sla} P. J. Slater, Locating dominating sets and locating-dominating sets, \emph{Graph theory, combinatorics, and algorithms}, Vol. 1, 2 (Kalamazoo, MI, 1992), 1073--1079, Wiley-Intersci. Publ., Wiley, New York, 1995.

\bibitem{TZTX} M. T. Thai, N.  Zhang, R. Tiwari, X. Xu, Xiaochun, On approximation algorithms of $k$-connected $m$-dominating sets in disk graphs, \emph{Theoret. Comput. Sci.} {\bf 385} (2007), no. 1--3, 49--59.

\bibitem{Vo} L. Volkmann, Connected $p$-domination in graphs, \emph{Util. Math.} {\bf 79} (2009) 81--90.

\bibitem{West} D. B. West, \emph{Introduction to graph theory} (second edition), Prentice Hall, Inc., Upper Saddle River, NJ, 2001, xx+588 pp.


\end{thebibliography}
\end{document}